\documentclass[10pt,journal]{IEEEtran}

\usepackage[utf8]{inputenc} 
\usepackage[T1]{fontenc}    
\usepackage{booktabs}       
\usepackage{amsfonts}       
\usepackage{nicefrac}       
\usepackage{microtype}      
\usepackage{amsmath}
\usepackage{amssymb}
\usepackage{bm}
\usepackage{esvect}
\usepackage[rgb]{xcolor}
\usepackage{amsthm}
\usepackage{etoolbox}
\usepackage{multicol}
\usepackage{scrextend}
\usepackage{xspace}
\usepackage{graphicx}
\usepackage{placeins}
\usepackage{caption}
\usepackage{listings}
\usepackage{textgreek}
\usepackage{ulem}

\newcommand{\rg}[1]{{\color{black}#1}}

\newtheorem{theorem}{Theorem}[section]
\newtheorem{corollary}{Corollary}[theorem]
\newtheorem{lemma}{Lemma}[section]
 
\newtheorem{df}{Definition}
\newtheorem{rmk}{Remark}

\DeclareMathOperator{\argmin}{argmin}

\newcommand{\OOO}{\mathcal{O}}
\newcommand{\NNN}{\mathcal{N}}
\newcommand\Aa{\arraycolsep=1.4pt\left[\begin{array}{cc}\mA&\va\end{array}\right]}

\newcommand\Bt{\tilde{\mB}}
\newcommand\BZZZ{\left[ \begin{array}{cc} \mB & 0 \\ 0 & 0 \end{array} \right]}
\newcommand{\norms}[1]{\left\Vert#1\right\Vert^2}
\newcommand{\ip}[1]{\left\langle #1 \right\rangle}
\newcommand{\ips}[1]{\langle #1 \rangle}

\newcommand\Ai{\mA_{\setminus i}}

\newcommand\Rj{R_{\sTj}}
\newcommand\Rq{R_{\sTq}}
\newcommand\Lt{\tilde{\sL}} 
\newcommand\Tt{\tilde{\sT}} 
\newcommand\yt{\tilde{\vec{y}}} 
\newcommand\xt{\tilde{\vec{x}}} 
\newcommand\xh{\hat{\vec x}} 
\newcommand\wh{\hat{\vec w}} 
\renewcommand{\vec}[1]{\bm{#1}}
\newcommand\At{{\tilde{\mA}}}

\newcommand{\mi}{{\setminus i}}

\newcommand\vrho{\vec{\rho}}
\newcommand\vy{\vec{y}}
\newcommand\vc{\vec{c}}
\newcommand\va{\vec{a}}
\newcommand\vr{\vec{r}}
\newcommand\vz{\vec{z}}
\newcommand\vm{\vec{m}}
\newcommand{\di}{\vm_{(i)}}
\renewcommand{\dj}{\vm_{(j)}}
\newcommand{\dq}{\vm_{(q)}}
\newcommand{\vo}{\vec{\omega}}
\newcommand{\oi}{\vo_{(i)}}
\newcommand{\oj}{\vo_{(j)}}
\newcommand{\vd}{\vec{d}}
\newcommand\vw{\vec{w}}
\newcommand\vx{\vec{x}}
\newcommand\ve{\vec{e}}
\newcommand{\vb}{\vec{b}}

\renewcommand{\v}{\vec{v}}
\newcommand{\vvv}{\v_{\sLam}}
\newcommand{\vvvi}{\v_{\sLam\cup i}}
\newcommand{\vq}{\vec{q}}

\newcommand{\valpha}{{\vec{\alpha}}}
\newcommand{\vgam}{\vec{\gamma}}
\newcommand\sett[1]{{#1}}
\newcommand{\sT}{{\sett{T}}}
\newcommand{\sTt}{{\sett{T}\cup i }}
\newcommand{\sTj}{{\sett{T}\setminus j }}
\newcommand{\sTq}{{\sett{T}\setminus q }}
\newcommand{\sL}{{\sett{L}}}

\newcommand{\sLam}{{\sett{\varLambda}}}
\newcommand{\sLami}{{\sLam \cup i}}
\newcommand{\sLammi}{{\sLam \setminus i}}

\newcommand{\matrx}[1]{\bm{#1}}
\newcommand{\RR}{\mathbb{R}}
\newcommand{\mA}{{\matrx{A}}}
\newcommand{\mB}{{\matrx{B}}}
\newcommand{\mD}{{\matrx{D}}}
\newcommand{\mI}{{\matrx{I}}}
\newcommand{\mM}{\matrx{M}}
\newcommand{\mS}{\matrx{S}}
\newcommand{\mF}{\matrx{F}}
\newcommand{\mC}{\matrx{C}}
\newcommand{\mO}{\matrx{\varOmega}}
\newcommand{\mG}{\matrx{\varGamma}}
\newcommand{\GGi}{\mG_{\sLam\cup i}}
\newcommand{\GG}{\mG_{\sLam}}
\newcommand{\mQ}{\matrx{Q}}
\newcommand{\QM}{\mQ_{\mM}}
\newcommand{\QO}{\mQ_{\mM^\perp}}
\newcommand{\OL}{\mO_\sLam}
\newcommand{\OLi}{\mO_\sLami}
\newcommand{\OLmi}{\mO_\sLammi}
\newcommand{\BL}{\mB_\sLam^{-1}}
\newcommand{\CL}{\mC_\sLam}
\newcommand{\Zer}{\matrx{0}}
\newcommand{\mL}{\matrx{L}}
\newcommand{\LL}{\mL_\sLam}
\newcommand{\LLi}{\mL_{\sLam\cup i}}

\newcommand{\mH}{\matrx{H}}

\newcommand{\Rn}{R_{\sT_n}}
\newcommand{\Rnn}{R_{\sT_{n+1}}}

\newcommand{\MMM}[2]{\left[\begin{array}{#1} #2 \end{array}\right]}

\newcommand{\ORIP}{$\Omega$-RIP\xspace}

\newcommand{\cosamp}{CoSaMP\xspace}

\newcommand{\vs}{}
\numberwithin{df}{section}

\expandafter\def\expandafter\normalsize\expandafter{%
	\normalsize
	\setlength\abovedisplayskip{5.0pt plus 2.0pt minus 3.0pt}
	\setlength\belowdisplayskip{5.0pt plus 2.0pt minus 3.0pt}
	\setlength\belowdisplayshortskip{3.0pt plus 2.0pt minus 2.0pt}
}

\title{Efficient Least Residual Greedy Algorithms for Sparse Recovery}   
\makeatletter 
\let\Title\@title
\makeatother
\author{\IEEEauthorblockN{Guy Leibovitz\IEEEauthorrefmark{1},
		Raja Giryes\IEEEauthorrefmark{2}\thanks{© 2020 IEEE.  Personal use of this material is permitted.  Permission from IEEE must be obtained for all other uses, in any current or future media, including reprinting/republishing this material for advertising or promotional purposes, creating new collective works, for resale or redistribution to servers or lists, or reuse of any copyrighted component of this work in other works. }}\\
	\IEEEauthorblockA{
			\IEEEauthorrefmark{1}guyleibovitz@mail.tau.ac.il,
			\IEEEauthorrefmark{2}{raja@tauex.tau.ac.il}\\ 
		School of Electrical Engineering,
		Tel-Aviv University\\
		Tel-Aviv, Israel
	}
}

\begin{document}
	\maketitle
	\begin{abstract}
We present a novel stagewise strategy for improving greedy algorithms for sparse recovery. We demonstrate its efficiency both for synthesis and analysis sparse priors, where in both cases we demonstrate its computational efficiency and competitive reconstruction accuracy. In the synthesis case, we also provide theoretical guarantees for the signal recovery that are on par with the existing perfect reconstruction bounds for the relaxation based solvers and other sophisticated greedy algorithms.

	\end{abstract}


\section{Introduction} \label{sec:intro} 
In this paper we consider the generic linear inverse problem of recovering a vector $\vx\in\RR^n$ from an incomplete set of  measurements $\vy \in \RR^m$ ($m < n$) available via 
\begin{equation}\label{eq:first}
\vy= \mM\vx+\vw,
\end{equation}
where $\mM \in \RR^{m \times n}$ is the measurement matrix, $\vw$ is an additive noise, and $\vx$ is assumed to have a parsimonious representation.   
 
 \subsection{The sparse synthesis model}
 The first parsimonious model we consider is the sparse synthesis framework. It assumes $\vx$ to be a sparse vector, i.e., $\|\vx\|_0= k$, $k\ll m$, where $\|\cdot\|_0$ is the $\ell_0$ pseudo-norm that counts the number of non-zeros in a vector. 
 This prior has received a lot of attention in the signal processing and statistics communities, and been used in areas such as regression \cite{tibshirani1996regression}, signal and image restoration \cite{bruckstein2009sparse}, and compressed sensing \cite{eldar2012compressed, Foucart13Mathematical}. 
 Finding the best $k$-term approximation of $\vy$ is proven to be NP-hard \cite{Natarajan95Sparse, Davis1997Adaptive}, 
 unless some regularity conditions applies to $\mM$ \cite{Donoho06Stable, Candes06Near, Rudelson06Sparse}. 
 
 A popular condition on $\mM$ is the Restricted Isometry Property (RIP) \cite{rip} that constrains its subsets of columns. 
 \begin{df}[RIP] \label{def:RIP}
 	A matrix $\mM$ satisfies the RIP of order $k$ with a constant $\delta_k\in(0,1)$ if   
 	\begin{equation}
 	(1-\delta_k) \|\v\|^2 \leq \|\mM_\sT\v \|^2 \leq (1+\delta_k) \|\v\|^2,
 	\end{equation}
 	holds $\forall |T| = k, \v \in \RR^k$, where the matrix $\mM_\sT$ is comprised of the columns of $\mM$ corresponding to the indices of a support $\sT$ of size $k$.
 \end{df} 	

One of the common methods for sparse signal estimation under this assumption is the $\ell_1$-regularized convex optimization \cite{basis,tibshirani1996regression}. 
This approach leads to a stable recovery of the sparse vector under some conditions on the RIP (e.g. the work in \cite{sharp} proves that $\delta_k< 1/3$ is a sharp bound). Yet, the $\ell_1$-method usually requires more computations than another widespread methodology, the greedy strategy. 

Two popular greedy algorithms are the orthogonal matching pursuit (OMP) and the orthogonal least squares (OLS)\footnote{ The OLS algorithm entertains a plethora of designations in the literature being rediscovered under different names several times. For example, in statistics it is known as forward stepwise regression \cite{miller2002subset}. Other names include least squares
orthogonal matching pursuit (LSOMP) \cite{elad2010sparse}, forward-selection, and order recursive matching pursuit (ORMP) \cite{eldar2012compressed} to name a few. See \cite{LSOMPdifference} for a discussion on the matter}.
 OMP and OLS obtain a representation by greedily selecting one atom of the dictionary $\mM$ at a time. OMP chooses the most correlated atom in $\mM$ to the residual error. OLS on the other hand, improves over OMP by picking the atom (column) that would yield the smallest approximation error when added to the chosen set. In the statistics literature, this is sometimes called forward-selection. However, OLS comes with an additional computational cost (order of $k$ times the OMP complexity). Optimized OMP (OOMP) \cite{optimized} has been proposed to accelerate OLS, but it requires storing and updating an additional copy of the dictionary. 
In this paper, we propose a more space and time efficient method for accelerating OLS. Indeed, other OLS acceleration strategies exist \cite{Hashemi18Accelerated, Li18Improving, Wang17Recovery, Wen17Nearly}. Yet, for completeness we present our acceleration variant for OLS below, as the tools we use to construct it serve us later to propose our efficient replacement based techniques that achieve better performance as we show hereafter.
 
 Both OLS and OMP use a one-off strategy, where an atom never leaves the selected support after it enters. One option for improving one-off programs is called back-tracing, the re-consideration of atoms in the selected support. Pursuit algorithms that use this approach include:      
\cosamp \cite{cosamp} and its signal space variants \cite{Davenport13Signal, Giryes15GreedySignal, Tirer17Generalizing}, Subspace Pursuit (SP) \cite{sp}, iterative hard thresholding (IHT) \cite{Blumensath09Iterative},  hard-thresholding pursuit (HTP) \cite{Foucart11Hard} and OMP with Replacement (OMPR) \cite{ompr}, which received a lot of attention for their $\ell_1$-like reconstruction guarantees, e.g., $ \delta_{3k} \leq 0.4859$ and $\delta_{4k} \leq 0.3843$ for SP and CoSaMP respectively, \cite{foucart2012sparse, song2014improved}.
Note that though the existing guarantees for OMPR are also based on the RIP, they impose an additional constraint on the mutual coherence of $\mM$ ($\delta_2$) and on a parameter of this algorithm that requires tuning.   
In \cite{Soussen11Bernoulli,Varadarajan11Stepwise}, OLS based algorithms with replacement have been proposed but without RIP based recovery guarantees. We believe that the theory developed here can be applied also to \cite{Soussen11Bernoulli,Varadarajan11Stepwise}.
\subsection{The analysis cosparse model}
\label{sec:analintro}
The second framework we consider is the more recent cosparse analysis model \cite{elad2007analysis}. The signal $\vx$ is assumed to be sparse after an analysis operator $\mO \in \RR^{p\times n}$ is applied to it. In this case $k=\|\mO \vx\|_0$ is the sparsity, and $l = p-k$ is the number of rows in $\mO$ that are orthogonal to $\vx$, denoted as the cosparsity of $\vx$. The subspace dimension in which $\vx$ resides is $n-\text{rank}(\mO_{\sLam})$ where $\sLam$ $ (|\sLam|= l)$ is the set of rows from the analysis operator that are orthogonal to $\vx$. In the analysis model, one strives to enlarge $l$, i.e. make the signal $\mO_{\sLam}\vx$ as sparse as possible. In this model, one aims at minimizing 
\begin{equation}
\label{eq:anall0}
\xh = \argmin_{\vx } \|\mO\vx\|_0 \text{  s.t.  }  \vy=\mM\vx. 
\end{equation} 
 For more details on this model, we refer the reader to \cite{candes2011compressed, elad2007analysis, giryes2014greedy,nam2013cosparse, Vaiter13Robust} and references therein.

Solving \eqref{eq:anall0} is also NP-hard \cite{nam2013cosparse, Tillmann14Projection}. Thus, approximation strategies have been proposed for \eqref{eq:anall0} as well. 
As in the synthesis model, a popular reconstruction technique is relaxing the $\ell_0$ pseudo-norm in \eqref{eq:anall0} to the $\ell_1$ norm. The work in \cite{giryes2014greedy} proposed analysis versions of SP, \cosamp, IHT and HTP. 
Another popular algorithm for the analysis model is the greedy analysis pursuit (GAP) \cite{nam2013cosparse}. It can be thought of as the analysis equivalent of OMP. Unlike OMP, GAP operates backwards instead of forward, i.e. it starts with all the rows of $\mO$ inside the cosupport and eliminates one row at a time until a cosupport of a desired size is reached or the norm of $\OL\xh$ is small, where $\sLam$ and $\xh$ are the current cosupport and estimate respectively. 
Having found the cosupport $\sLam$, GAP uses it to approximate the solution by solving the following optimization problem.
 
\begin{df}[Relaxed analysis objective function]  
	\label{df:relaxedanal}
	Given a set of indices $\sLam$, where $\mO_{\sLam}$ is a matrix comprised of the rows of $\mO$ indexed by $\sLam$, the quantity GAP seeks to minimize is   
	\begin{equation}
	\label{eq:P1}
	\xh = \arg\min_{\vx} \norms{\mO_\sLam\vx} \text{ s.t. }  \norms{\vy-\mM\vx} = 0. 
	\end{equation} 
\end{df}

Recovery guarantees also exist for the analysis model, but unlike the synthesis case, having $\mM$ that satisfies a restricted isometry condition is not enough and a restriction on $\mO$ is also required \cite{candes2011compressed,giryes2014greedy, liu2012compressed, needell2012total}. In \cite{nam2013cosparse}, the relationship between minimizing \eqref{eq:P1} and solving \eqref{eq:anall0} has been studied.
Recovery guarantees have been developed for the $\ell_1$ strategy for the cases that $\mO$ is a frame (see e.g. \cite{candes2011compressed,liu2012compressed}) or the finite difference operator (e.g. \cite{needell2012total}), among else. For the greedy algorithms, similar guarantees have been developed under the assumption that $\mO$ is a frame \cite{Giryes16Greedy} or that some near-optimal projection exists for it \cite{giryes2014greedy}.

\subsection{Paper contribution}

In this work, we utilize the OLS selection heuristic. It can be thought of as a general approach where an item is selected based on its effect on the residual rather than relying on a simple correlation with that residual. We propose new greedy synthesis and analysis methods for sparse approximation that resemble OMP and GAP in their computational complexity but have better theoretical guarantees than OMP and empirical reconstruction performance than both of them. The proposed strategy resembles other efficient methods that have been recently proposed such as OMPR \cite{ompr}. It attains a competitive theoretical guarantees compared to it: on the one hand the RIP condition of OMPR is slightly better than the one we derive for our algorithms, yet, on the other hand OMPR requires an additional (mild) assumption on the coherence of the dictionary, which is not required by our guarantees.

Our proposed pursuits are designed such that their stopping criteria is either the desired sparsity level or the target residual error similar to OMP, OLS, and GAP. 


Due to the fact that our methods are based on the OLS criterion for selecting (and removing) atoms, which is known to be more
resilient to cross correlation between dictionary atoms (see \cite{soussen2013joint} for a comparison to the OMP criterion), 
they are more resilient to correlations in the dictionary. We supply RIP based guarantees for these methods when the measurement is noise-free or corrupted by Gaussian noise. While for OLS, the best known guarantees require $m =O(k^2)$ measurements for getting perfect reconstruction guarantees \cite{Wen17Nearly},  for our proposed schemes it is required to have only $m = O(k\log(k/n))$.

Equipped with the notion that the OLS heuristic is advantageous with highly correlated dictionaries, we derive an algorithm for the analysis model which is equivalent to OLS, and improve it further by using backtracking to achieve better estimation. 
Simulations demonstrate the advantage our algorithms have in various scenarios.

\subsection{Organization} 
The remainder of this paper is organized as follows.
Section \ref{sec:synt} contains the synthesis algorithms and their properties. It briefly describes the OLS algorithm and presents some preliminary lemmas that aid speeding-up the calculations in the techniques presented. 
Then it introduces our proposed methods and demonstrates empirically their supremacy in various scenarios. 
In Section~\ref{sec:theory_synthesis}, we develop their theoretical performance guarantees. 
Section~\ref{sec:anal} focuses on the analysis model. It proposes a novel efficient analysis OLS-like greedy strategy with replacement. This technique is shown to have superior performance to other programs designed for the analsys framework. 
Section~\ref{sec:conc} concludes the paper.

\subsection{Notation} 
We summarize here the notation used in this paper: 
	$\va$ is a vector, $\vv{\va}=\va/\|\va\|$ is its normalized version, $a$ is a scalar, and $\mA$ is a matrix. $\va{(i)}$ is the $i$th entry of $\va$; $\va_{\mi}$ is the vector $\va$ without entry $i$; $\mA_\mi$ is $\mA$ without the $i$th column; and the $i$th column and row are deleted from $\mA$ to get $\mA_{\mi,\mi}$; 
	 $\|\cdot\|$ is the $\ell_2$ norm; and
	$\di$ designates the $i$th atom of the dictionary $\mM\in \RR^{m\times n}$ (and we assume $\|\di\| = 1\ \ \forall i$). $\vx\in \RR^n$ is the unknown vector with sparsity $k = \|\vx\|_0$ in the synthesis case and $k = \|\mO \vx\|_0$ in the analysis case, and $\vy = \mM\vx+\vw$ is the signal we have.   
	Uppercase non-bold letters are sets of indices (e.g. $\sT$). Unless stated otherwise, $\mA=\mM_\sT$ denotes the sub-matrix of $\mM$ made of the \emph{columns} indexed by $\sT$, whereas $\mO_\sLam$ is comprised of the sub-matrix of $\mO$ made of the \emph{rows} indexed by $\sLam$. \rg{This has only one exception in the paper, namely, $\LL$, which contains a subset of columns of $\mL\triangleq \QO^T\mO^T$.}
The estimate of $\vx$ using the atoms in $\sT$ is denoted as $\xh_\mA=\xh_\sT=(\mA^T\mA)^{-1}\mA^T\vy$; the orthogonal projection onto the column space of $\mA$ as $P_\mA=P_\sT= \mA(\mA^T\mA)^{-1}\mA^T$; and the orthogonal complement $R_\mA=R_\sT= I-P_\mA$. In the analysis part of the paper, the projection is defined onto the \emph{row} span.
	Finally, $\odot$ represents an element-wise multiplication.

\section{Efficient least-residual techniques for the synthesis model} 
\label{sec:synt}
\subsection{Preliminaries}
\label{ssec:pre} 
The following are known preliminary lemmas that will aid us in the derivation of two new greedy techniques as well as their theoretical recovery guarantees. We start with two variants of the matrix inversion lemma for a column addition and deletion that follow from a straight forward application of the formula for the inverse of a two by two block matrix using the Schur complement.  
		\begin{lemma}[Matrix inversion lemma for column addition] \label{lem:intro:bckgrnd:shermor}
Let $\mB = (\mA^T \mA)^{-1}$ and $\tilde{\mA} = \Aa$. Then we may calculate $\tilde{\mB} = (\tilde{\mA}^T \tilde{\mA})^{-1}$ as follows:
		\begin{equation}
		\Bt = \BZZZ + {1\over r} \left[ \begin{array}{c}\hat{\ve} \\ -1 \end{array} \right] \left[ \begin{array}{cc}\hat{\ve}^T & -1 \end{array} \right], \label{eq:shermor}
		\end{equation}
		where $\hat{\ve} = \mA^\dagger \va$, and $r = \|R_\mA\va\|^2$.
		\end{lemma} 
A straight forward consequence of Lemma \ref{lem:intro:bckgrnd:shermor} is the following update for column removal:
		\begin{lemma}[Matrix inversion lemma for column removal] \label{lem:intro:bckgrnd:shermor_rem}
			Let $\tilde{\mB} = (\tilde{\mA}^T \tilde{\mA})^{-1}$ with $\tilde{\mA} = \Aa$. Then we may calculate $\mB = (\mA^T \mA)^{-1}$ as follows:
			 	\begin{align}
			 	\hat\ve = -\tilde{\mB}_\mi(:,i),\ \ r = \tilde{\mB}(i,i)^{-1},\ \ \mB = \tilde{\mB}_{\mi,\mi} - r\hat\ve\hat\ve^T \label{eq:col_rem}.
			 	\end{align} 
		\end{lemma}
		Note that in order to use Lemma \ref{lem:intro:bckgrnd:shermor} for column insertion at a general location in $\mA$, simply insert $\va$ at the last index and permute $\At$ and $\Bt$ afterwards. We present now several other helpful lemmas. Their proofs are deferred to Appendix~\ref{sec:proofs}.

\begin{lemma}[Error change]
					\label{lem:intro:bckgrnd:coladderr}
					 Let the estimation error of $\vy$ using $\mA$ be $\norms{\vy - \mA\xh_{\mA}} =  \norms{R_\mA \vy} $, and let $\At = \Aa$. Then the difference in norm of the residual before and after the addition of $\va$ to the support, may be written as: 
						\begin{equation}
						\ips{\vv{R_\mA\va},\vy}^2 =  \|R_{\mA}\vy\|^2-\|R_{\At}\vy\|^2 = \|P_{\At}\vy\|^2-\|P_{\mA}\vy\|^2
						\label{eq:coll_add:resid}.
						\end{equation}
		\end{lemma}

\begin{lemma}[The value of $\xh_\mA (i)$]
	\label{lem:intro:bckgrnd:xhatival}
	The least squares estimate of $\xh_\mA = (\mA^T\mA)^{-1}\mA^T \vy $ can be written as the representation of $\vy$ in a bi-orthogonal basis for the space spanned by the columns of $\mA$, i.e., its $i$-th entry is 
	\begin{align} \xh_\mA(i)  =  \ip{R_{\Ai} \di,\vy}/{ \|R_{\Ai} \di\|^2}.\label{eq:xh_val}
	\end{align}
\end{lemma}	

The following lemma is an interesting consequence of the ones above.  
   \begin{lemma} 		\label{lem:intro:bckgrnd:leastcol} 
	Let $\xh = \tilde{\mB}\tilde{\mA}^T\vy$, with $\Bt = (\At^T\At)^{-1}$, then the least contributing column $i$ of $\tilde{\mA}$ \rg{(i.e., the column whose removal has the smallest effect on the error)} for the least squares estimate of $\vx$ from $\vy$ is the one whose index corresponds to 
		\begin{align} 
		\arg \min_i {\xh(i)^2/\tilde{\mB}(i,i)}. 
		\label{eq:leastcontcol} 
		\end{align} 
   \end{lemma}

Using the above lemmas, we are now equipped to describe OLS and the differences it bears to OMP. 
\subsubsection{Orthogonal least squares (OLS)}
Lemma \ref{lem:intro:bckgrnd:coladderr} implies that given a set of columns, $\mA=\mM_\sT$, used for the estimation of a signal, adding to it a column that satisfies $\vd = \arg \max_{\vd\in \mM} \ips{\vy,\vv{R_\mA \vd}}^2$ yields the smallest residual among all atoms in the dictionary. This notion is the basis for OLS. Notice that in OMP (and other methods such as SP, \cosamp and OMPR) the selection criterion is based on $\arg \max_{\vd\in \mM} \ips{\vy,{R_\mA \vd}}^2$, lacking the normalization by $\|R_\mA \vd\|^2$. OLS enlarges its selected support iteratively, where at each step a new atom that satisfies Lemma  \ref{lem:intro:bckgrnd:coladderr} is selected. The pseudo-code for OLS can be found in Algorithm \ref{tb:OLS}. 

\begin{table}[t]		
\caption{Orthogonal Least Squares (OLS)}		
\label{tb:OLS}		
\centering		
\begin{tabular}{l}		
	\toprule		
	Input: dictionary $\mM$, measurement $\vy$, target cardinality $k$ or error $\epsilon_t$ \\
	Output:  $\xh$ with $k$ elements or $\epsilon_t$ residual error, and $\sT$ its support  		\\
	\midrule
		\textbf{init}\\	
	\quad $T\leftarrow\{\}$, $\epsilon_0 \leftarrow \|\vy\|^2$, $\vr\leftarrow \vy$		\\
	\textbf{while} $|T|<k$ or $\epsilon_0>\epsilon_t$ 		\\
	\quad $i \leftarrow \arg \max_i \left\{ \ip{\vr,\di}^2 / \|R_\sT\di\|^2\right\} $  		\\
	\quad $\sT \leftarrow \sT\cup \{i\}$  		\\
	\quad $\epsilon_0 \leftarrow \epsilon_0 -   \ip{\vr,\di}^2 / \|R_\sT\di\|^2 $		\\
	\quad $\vr\leftarrow R_\sT\vy$		\\
	\textbf{end while }		\\
	\textbf{return} $\sT$, $\xh=\mM_\sT^\dagger\vy$		\\
	\bottomrule		
\end{tabular}		
\end{table}			
		\subsection{Proposed algorithms} \label{sec:alg}
We introduce two new algorithms that use a similar approach to OLS, by selecting atoms for inclusion/exclusion based on $\langle\vv{R_\sT \di},\vy\rangle$ as the metric. Before describing them, we \rg{present} first an accelerated version of the standard OLS. Our acceleration is based on the matrix inversion lemma \rg{(similar to a strategy presented in \cite{Sturm12Comparison})}. Note that this is not the only possible acceleration technique; for example, some methods use the QR or Cholesky factorization \cite{Mairal10Online,miller2002subset,Sturm12Comparison}. 

\subsubsection{Fast orthogonal least squares (FOLS)}
\rg{The main computational burden of regular OLS as presented in Algorithm~\ref{tb:OLS} is that in each iteration it requires calculating the projection over all the atoms. Our goal is to perform these calculations efficiently using the fact that we only add one atom to the support at each iteration. Thus, we use the matrix inversion lemma to get the FOLS method in Algorithm \ref{tb:FOLS}. It calculates OLS using a single dictionary application.

In this algorithm,} we introduce two length $n$ vectors with an efficient update scheme that requires a single dictionary-vector multiplication. To formulate this update scheme recall Lemma \ref{lem:intro:bckgrnd:coladderr}. We calculate $\langle \vv{R_\sT \di} , \vy \rangle^2$ for each of the atoms in the dictionary by dividing the square of the elements of $\vc(i) = \langle R_\sT\di,\vy \rangle$ by the elements of $\vrho(i) = \|R_\sT\di\|^2$. Upon the addition of $\di$ to the support, the updates of $\vc$ and $\vrho$ relies on the following equalities:
 \begin{align} 
 \vc(i) &= \ip{R_{\sT'}\di,\vy}  = \ip{R_\sT\di,\vy} - \ip{\vv\v,\di}\ip{\vv\v,\vy}  \nonumber\\ 
 &= \vc_{prev}(i) - {\tilde \vrho (i) \vc_{prev}(i)/ \|\v\|^2}\label{eq:LSOMP:c_update}\\
 \vrho(i) &= \|R_{\sT'}\di\|^2 =  \ip{R_\sT\di,\di} - \ip{\vv\v,\di}^2\nonumber\\
 &= \vrho_{prev}(i) - {\tilde{\vrho}(i)^2/ \|\v\|^2},  \label{eq:FLSOMP:rho_update}
 \end{align}  
  where we denote $\vc_{prev}$, $\vrho_{prev}$ the vectors before the update, $i$ the index of the atom that enters the support, and
  $\v = R_\sT\di$, $\tilde\vrho = \mM^T\v$, $\sT'=\sT\cup i$. Notice that these operations makes FOLS more efficient than OLS as they spare the inversion operation at each iteration.
 \begin{table}[t]	
 	\caption{Fast Orthogonal Least Squares (FOLS)}	
 	\label{tb:FOLS}	
 	\centering	
 	\begin{tabular}{l}	
 		\toprule	
 		Input: dictionary $\mM$, measurement $\vy$, target cardinality $k$ or error $\epsilon_t$ 	\\
 		Output:  $\xh$ with $k$ elements or $\epsilon_t$ residual error, and $\sT$ its support  	\\
 		\midrule
 		\textbf{init}\\	
 		\quad $\sT\leftarrow\{\}$, $\epsilon_0 \leftarrow \|\vy\|^2$, $\vrho \leftarrow \textbf{1}_{n\times 1}$, $\vc \leftarrow \mM^T\vy$ 	\\
 		\textbf{while} $|\sT|<k$ or $\epsilon_0>\epsilon_t$  \quad \quad  \textbf{-Only one condition is tested}	\\
 		\quad$i \leftarrow \arg \max_{i\notin \sT} {\vc(i)^2/ \vrho(i)}$	
\textbf{-Atom that reduces the error the most
} 		
 		\\
 		\quad$\sT \leftarrow \sT\cup \{i\}$  \quad \quad 
 				 \textbf{-Add atom to the support}\\
 		\quad$\epsilon_0 \leftarrow \epsilon_0 -   \vc(i)^2/\vrho(i) $ \quad \quad \textbf{-Update the residual}	\\	
 		\quad$\v \leftarrow R_\sT\di$ \quad \textbf{-Auxiliary}	\\
 		\quad$\tilde\vrho  \leftarrow \mM^T\v$       \quad \quad \textbf{-Auxiliary}                  	\\
 		\quad$\vc \leftarrow \vc - {\vc(i)\over \|\v\|^2}\tilde\vrho $         \quad \quad \textbf{-Auxiliary}                  	\\
 		\quad$\vrho \leftarrow \vrho - {1\over \|\v\|^2}\tilde\vrho\odot \tilde{\vrho}$  \quad  \quad \textbf{-Auxiliary}                   	\\
 		\textbf{end while}	\\
 		\textbf{return} $\sT$, $\xh = \mM_\sT^\dagger\vy$	\\
 		\bottomrule	
 	\end{tabular}	
 \end{table}	
\subsubsection{Orthogonal least squares with replacement (OLSR)}  \label{sec:alg:olsr}
Having a fast version for OLS, we propose now a novel algorithm that is based on it that we name
OLS with replacement (OLSR). \rg{The idea behind OLSR is very simple. It generates an initial support of size $k+1$, e.g., using FOLS, and then keeps replacing the least contribution atom (to reducing the error) in its current approximated support with the most contributing atom outside of the support. This technique stops when the atom outside the support contributes less than the atom that has been just removed. Thus, ending with a reconstruction $\xh$ with support of size $k$.}



	\begin{table}[t]	
	\caption{OLS with Replacement (OLSR)}	
	\label{tb:OLSR}	
	\centering	
	\begin{tabular}{l}	
		\toprule	
		Input: dictionary $\mM$, measurement $\vy$, cardinality $k$ ~~~~~~~~~~~~~~~~~~~~	\\
		Output:  $\xh$ \rg{of (at most)} $k$ elements, and $\sT$ its support \\
		\midrule	
		$\{\sT,\vc,\vrho,\epsilon_0\}\leftarrow $FOLS$(\mM,\vy,k+1)$	\\
		\textbf{if} $\epsilon_0 \approx 0 $ and $|T| \le k$\\
		\quad\textbf{return} $\sT$, $\xh = \mM_\sT^\dagger \vy$	\\
		\textbf{init}	\\
		\quad $\vc^{0} \leftarrow \mM^T\vy$, $\mB = (\mM_\sT^T\mM_\sT)^{-1}$, $\xh \leftarrow \mB\vc^{0}_\sT$	\\
		\textbf{loop}	\\
		\quad	$j \leftarrow \arg \min_{j\in \sT} \left\{\xh(j)^2/\mB(j,j)\right\}$   \textbf{-Find least contributing atom}	\\
		\quad	$\{\vc,\vrho\}\leftarrow\textit{updRem}(\mM,\mB,\sT,\xh,\vc,\vrho,j)$\quad (Alg. \ref{tb:upd})\\
		\quad	remove $\dj$ from $\mB$ using \eqref{eq:col_rem} \quad \textbf{-Matrix inversion lemma}	\\
		\quad	$\sT\leftarrow \sT \setminus \{j\}$ \quad \textbf{-Remove atom from the support}	\\
		\quad	$i \leftarrow \arg \max_{i\notin \sT} \left\{\vc(i)^2/ \vrho(i)\right\}$	 \textbf{-Atom most reducing the error}  \\
		\quad	\textbf{exit loop if} $\vc(i)^2/ \vrho(i) \geq \xh(j)^2/\mB(j,j) $	\textbf{Check if replacing helps}
		\\
		\quad	update $\mB$ with $\di$ using \eqref{eq:shermor}	 \quad \textbf{-Matrix inversion lemma}\\
		\quad	$\sT \leftarrow \sT\cup \{i\} $   \quad \textbf{-Add atom to the support}	\\
		\quad	$\xh \leftarrow \mB\vc^{0}_\sT $	 \quad \textbf{-Calculate current estimate to solution}\\
		\quad	$\{\vc,\vrho\} \leftarrow \textit{updAdd}(\mM,\mB,\sT,\xh,\vc,\vrho)$\quad (Alg. \ref{tb:upd})\\
		\textbf{end loop}	\\
		\textbf{return} $\sT$, $\xh$	\\
		\bottomrule	
	\end{tabular}	
\end{table}	
\subsubsection{Iterative orthogonal least squares with replacement (IOLSR)} \label{sec:alg:iolsr}
The second improvement introduced in this paper to OLS (which has the same computational cost) is IOLSR. It resembles Efroymson’s algorithm (\cite{miller2002subset}, Section 3.3) for the popular stepwise regression. At each step a new atom enters the support according to the regular OLS selection rule \eqref{eq:coll_add:resid}, and a test is performed to see if taking out one of the other atoms in the support will lower the residual compared to the beginning of the iteration (i.e. if the least-contributing column \eqref{eq:leastcontcol} is different than the one just added). If yes, a column is removed from the selected support. Otherwise, the selection set length is enlarged by 1.  

\subsubsection{\rg{Efficient implementation for OLSR and IOLSR}} 

\rg{Both OLSR and IOLSR require a single multiplication by $\mM$ at each loop iteration, similar to OMP. 
In general, this costs $\OOO(mn)$ flops, but for specific operators it might be lower, e.g., for DCT it is $\OOO(n\log n)$. Moreover, calculating the least and most contributing atoms directly is computationally demanding (as one needs to remove/add each candidate atom and then perform least square approximation with it only for calculating the error it produces in the reconstruction).
To facilitate these accelerations, we introduce two length $n$ vectors similar to FOLS, as described in \eqref{eq:LSOMP:c_update} and \eqref{eq:FLSOMP:rho_update}. 
While in FOLS we have only added elements to $\sT$, OLSR and IOLSR require also the removal of atoms from $\sT$, which is achieved in a similar way: The addition operator in \eqref{eq:LSOMP:c_update} and \eqref{eq:FLSOMP:rho_update} is replaced by  subtraction because by subtracting an element the residual grows and the subspace spanned by $\mM_\sT$ is reduced. Algorithm~\ref{tb:upd} describes how to perform both procedures efficiently. It is based on \eqref{eq:LSOMP:c_update}, \eqref{eq:FLSOMP:rho_update}, and lemmas \ref{lem:intro:bckgrnd:xhatival} and \ref{lem:intro:bckgrnd:leastcol} relying on the updated value of $\mB$ and $\xh$ from the previous iteration. We turn now to introduce our accelerations to both algorithms.}

\rg{ Algorithm \ref{tb:OLSR} presents our efficient implementation for OLSR. 
We provide here an explanation for its steps.
After producing a support estimation of size $k+1$ using FOLS, it efficiently removes and adds atoms. The least contributing atom is found using the formula in Lemma \ref{lem:intro:bckgrnd:leastcol}. The found atom is then removed using the matrix inversion lemma for column removal (Lemma~\ref{lem:intro:bckgrnd:shermor_rem}). To do this efficiently, we keep some auxiliary variables in the memory, which are updated using the ``updRem'' procedure from Algorithm~\ref{tb:upd}. Then, if an atom exists such that the residual after its inclusion in the support will be lower than the beginning of the iteration, it is inserted to the chosen set. The algorithm stops when no atom in the dictionary satisfies this condition, ending up with a support of size $k$. The steps for selecting the new atom that reduces the error the most and adding it are the same ones used in Algorithm~\ref{tb:FOLS} for FOLS (i.e., seeking the atom that minimizes the condition in Lemma~\ref{lem:intro:bckgrnd:leastcol}). The auxiliary variables used for the column addition in OLSR are updated in the ``updAdd'' procedure in Algorithm~\ref{tb:upd}.}

\rg{We also provide an efficient implementation for IOLSR in Algorithm \ref{tb:IOLSR}. At each iteration of this algorithm we first find the atom that reduces the error the most and add it (in the same way as in OLSR). We calculate the updated estimate of the solution with the updated support that includes now the new atom. Then we look for the least contributing atom in the support (using Lemma~\ref{lem:intro:bckgrnd:leastcol}). If it is the same as the atom just added then we continue to the next iteration. Otherwise, we remove the least contributing atom using Lemma~\ref{lem:intro:bckgrnd:shermor_rem} (we also update the auxiliary variables as we have done in OLSR). The stopping criterion for the algorithm is either support size based, i.e., the algorithm stops when the size of the support is $k+1$ (greater then $k$), or error based. In the first case, after the final iteration we remove one atom from the support to make its size $k$. We perform the final reconstruction using least squares with the updated support. 

}

	\begin{table}[t]	
	\caption{Iterative OLS with replacement (IOLSR)}	
	\label{tb:IOLSR}	
	\centering	
	\begin{tabular}{l}	
		\toprule	
		Input: dictionary $\mM$, measurement $\vy$, target cardinality $k$ or error $\epsilon_t$	\\
		Output:  $\xh$ with $k$ elements or $\epsilon_t$ residual error, and $\sT$ its support 	\\
		\midrule
		\textbf{init:}	\\
		\quad $\sT\leftarrow\{\}$, $\epsilon_0 \leftarrow \|\vy\|^2$ , $\vrho \leftarrow \bm{1}_{N\times 1}$, $j\leftarrow 1$, $\vc^{0} \leftarrow \mM^T\vy$, $\vc \leftarrow \vc^{0}$	\\
		\textbf{while} $\left(|\sT|<k + 1\right)$ \textbf{or} $\left(\epsilon_0>\epsilon_t \right)$   \textbf{-Only one is tested}		\\
		\quad $i \leftarrow \arg \max_{i\notin \sT} \left\{\vc(i)^2/ \vrho(i)\right\}$   \textbf{-Atom most reducing the error}		\\
		\quad update $\mB$ with $\di$ according to \eqref{eq:shermor}	\quad  \textbf{-Matrix inversion lemma}  \\
		\quad $\sT \leftarrow \sT\cup \{i\} $  \quad   \textbf{-Add atom to the support}	  	\\
		\quad $\xh \leftarrow \mB\vc^{0}_\sT $ \quad   \textbf{-Calculate current estimate to solution}		\\
		\quad $\epsilon_0 \leftarrow \epsilon_0 -   \xh(|\sT|)^2/\mB(|\sT|,|\sT|)  $   \quad \textbf{Update residual} 	\\
		\quad $\{\vc,\vrho\} \leftarrow \textit{updAdd}(\mM,\mB,\sT,\xh,\vc,\vrho)$\quad (Alg. \ref{tb:upd})\\
		\quad $j \leftarrow \arg \min_{j} \left\{\xh(j)^2/\mB(j,j)\right\}$  \textbf{-Find least contributing atom} 	\\
		\quad \textbf{if}  $j \neq |\sT|$  \textbf{Check if $j$ is the atom just added}	\\
		\quad\quad $\{\vc,\vrho\}\leftarrow\textit{updRem}(\mM,\mB,\sT,\xh,\vc,\vrho,j)$\quad (Alg. \ref{tb:upd})\\
		\quad \quad$\epsilon_0 \leftarrow \epsilon_0 +   \xh(j)^2/\mB(j,j)  $  \quad \textbf{-Update residual}	\\
		\quad \quad remove $\dj$ from $\mB$ using \eqref{eq:col_rem}	\quad  \textbf{-Matrix inversion lemma} \\
		\quad \quad$\sT\leftarrow \sT \setminus \{j\}$	\quad  \textbf{-Remove atom from the support} \\
		\quad \textbf{end if}	\\
		\textbf{end while}	\\
		\textbf{if} $|\sT|>k$ \textbf{-Only for the case of $k$ based stopping criterion}	\\
		\quad\quad perform: $\sT\leftarrow \sT \setminus \{j\}$ \textbf{-Remove atom from the support} 	\\
		\textbf{return} $\sT$, $\xh=\mM_\sT^\dagger \vy$	\\
		\bottomrule	
	\end{tabular}	
\end{table}	

\subsubsection{IOLSR and OLSR properties} 
\label{sec:IOLSR_OLSR_prop}

\begin{table*}[t]		
	\caption{Fast update procedures for the auxiliary vectors in OLSR and IOLSR}		
	\label{tb:upd}		
	\centering		
	\begin{tabular}{ll}		
		\toprule	
\begin{minipage}{0.37\textwidth}
$\mathrm{updAdd}(\mM,\mB,\sT,\xh,\vc,\vrho)$
\begin{addmargin}{0.3cm}
	$ \v \leftarrow \mM_\sT\mB(:,|T|)$\\
	$\tilde{\vrho}\leftarrow \mM^T\v$\\
	$\vc \leftarrow \vc - (\xh(k)/\mB(\tau,\tau))\tilde\vrho$\\
	$\vrho \leftarrow \vrho - (1/\mB(\tau,\tau))\tilde\vrho\odot \tilde\vrho$\\
	\textbf{return} $\vc,\vrho$
\end{addmargin}
\end{minipage}
&
\begin{minipage}{0.37\textwidth}
$\mathrm{updRem}(\mM,\mB,\sT,\xh,\vc,\vrho,j)$
\begin{addmargin}{0.3cm}
$ \v \leftarrow \mM_\sT\mB(:,j)$\\
$\tilde{\vrho}\leftarrow \mM^T\v$\\
$\vc \leftarrow \vc + (\xh(k)/\mB(\tau,\tau))\tilde\vrho$\\
$\vrho \leftarrow \vrho + (1/\mB(\tau,\tau))\tilde\vrho\odot \tilde\vrho$\\
\textbf{return} $\vc,\vrho$
\end{addmargin}
\end{minipage}\\
		\bottomrule		
	\end{tabular}		
\end{table*}

%

IOLSR and OLSR share the same performance guarantees in theorems \ref{thm:noiseless_delta} and \ref{thm:awgn}. We provide convergence speed analysis for OLSR in Theorem \ref{thm:olsr:convergence}. Note that in simulations IOLSR execution time demonstrates linear dependency on $k$ with a similar slope as OMP and OLSR. Another difference worth noting between IOLSR and OLSR, as well as other methods that enable backtracking such as subspace pursuit \cite{sp} or \cosamp \cite{cosamp}, is that IOLSR is able to run with a target residual norm (designated $\epsilon_t$) similarly to OMP and OLS, rather than target sparsity. Target residual as a stopping condition is 
usually more preferable in many applications leading to better results \cite{elad2010sparse}. 

The sparse approximation provided by IOLSR is usually better than the one of OLSR. This is due to the fact that the IOLSR subspace is sequentially optimized (i.e. the back tracing is performed on subspaces of increasing sizes rather than on a subspace of constant size), in comparison to OLSR and other methods that begin with a given support. 

\subsection{Synthesis model numerical experiments} \label{sec:sim}
\begin{figure}
	\centering
	\includegraphics[width=0.49\linewidth]{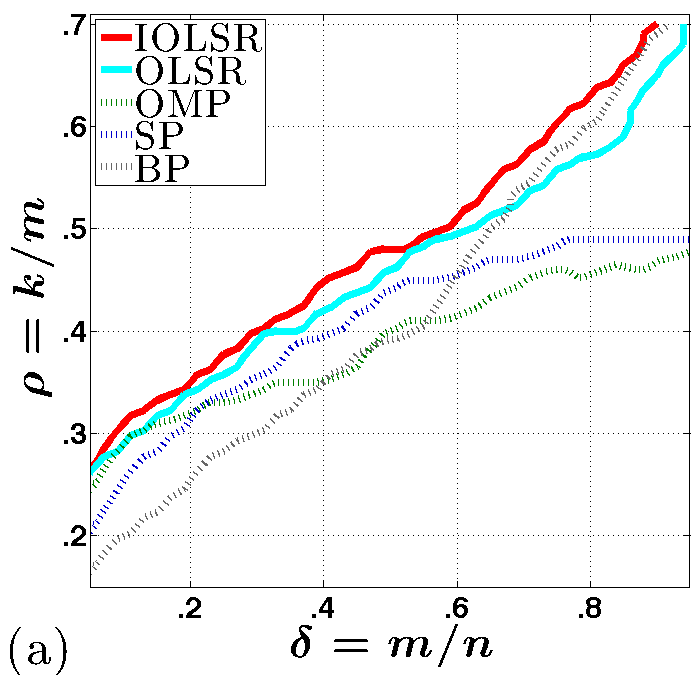}
	\includegraphics[width=0.49\linewidth]{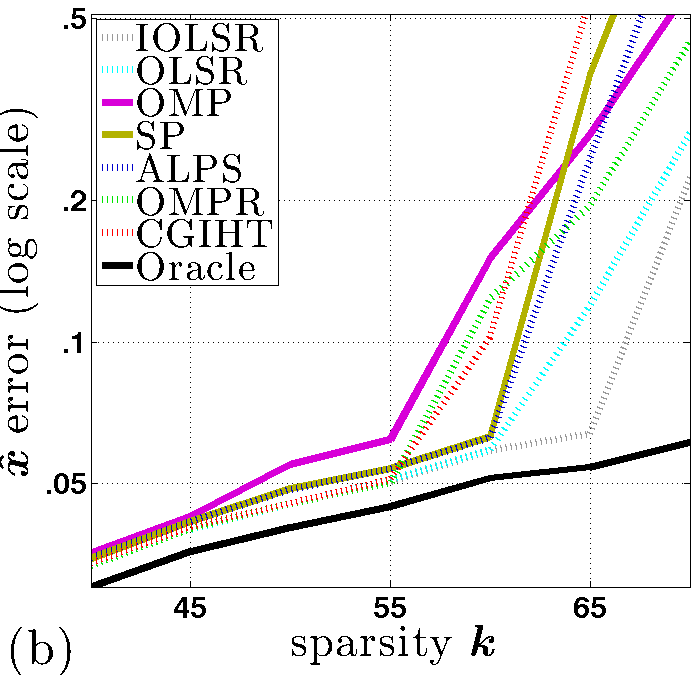}
	\includegraphics[width=0.49\linewidth]{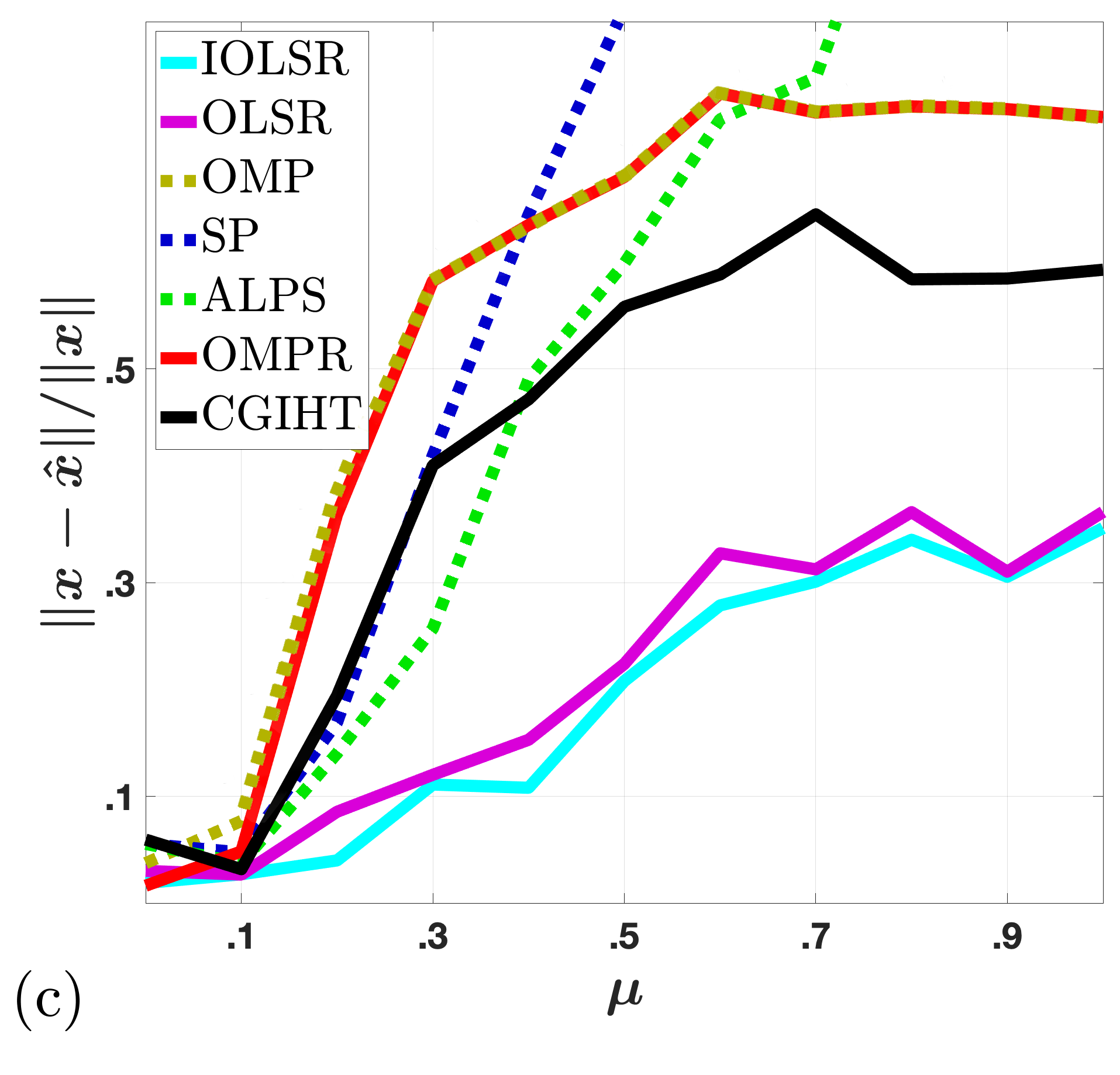}
	\includegraphics[width=0.49\linewidth]{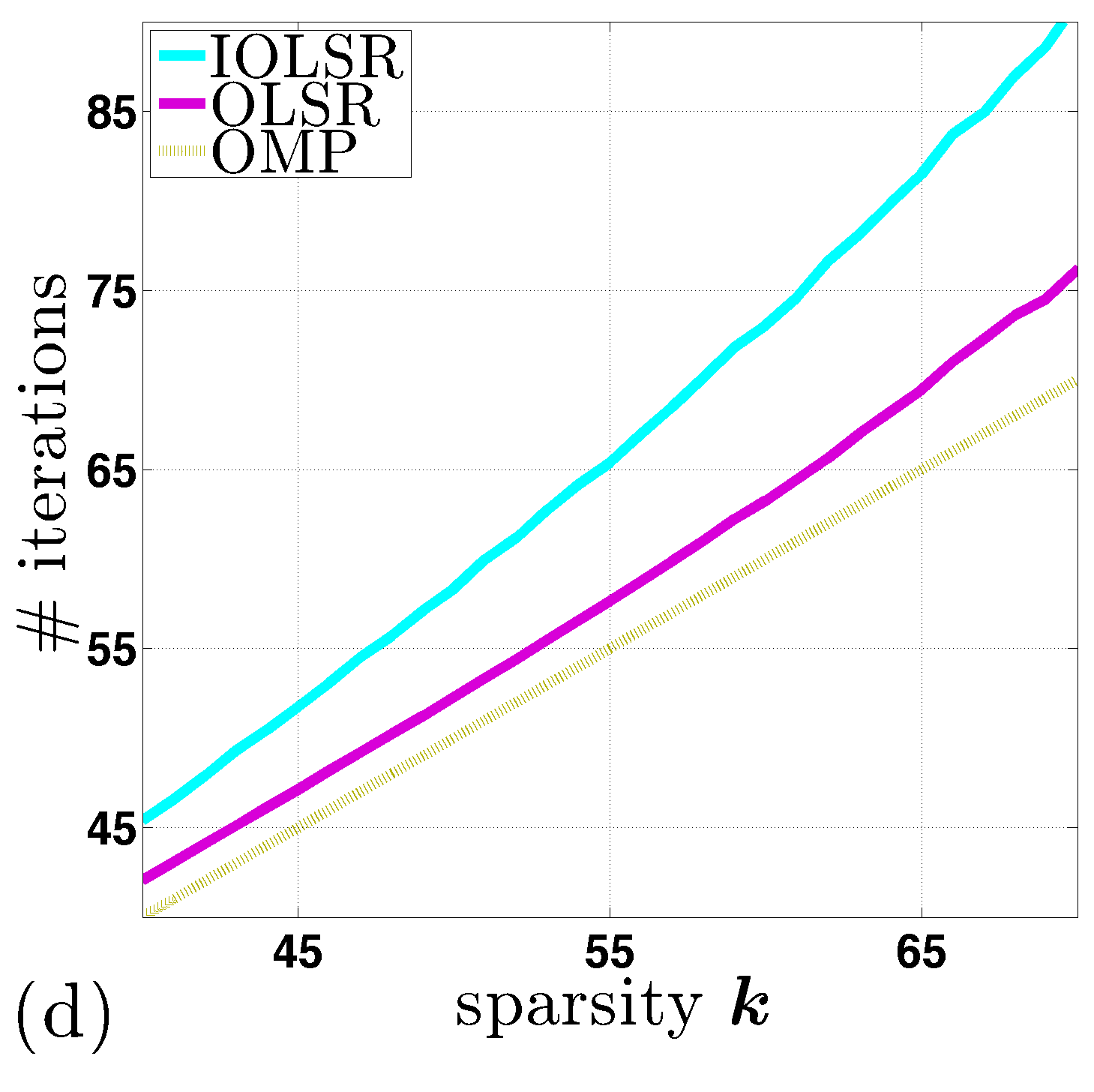}

	\captionsetup{singlelinecheck=off}
	\caption[Synthesis experiments]{
(a) Phase transition diagram (\cite{phasedonoho}) for $m=400$. Areas below the lines of each method represent an error threshold of $\|\xh-\vx\|^2/\|\vx\|^2 \leq 10^{-4}$ (higher lines are better).

(b)  Error ($\|\xh -\vx\|$) vs. sparsity in the presence of noise $\sigma = 0.01\|\vy_0\|/\sqrt{m}$ (yielding an average noise amplitude of 1\% of $\|\vy_0\|$) compared to an oracle that knows the true support; $n=600,\ m=200$.

(c) Error as a function of a coherency damaging parameter $\mu$; $n=300,\ m=100,\ k=30$.

(d)  Number of iterations of OLSR and IOLSR compared to OMP in the experiment in (b).
}
	\label{fig:phase}
\end{figure}
We turn now to numerically evaluate the performance of our methods\footnote{Matlab routines available at {web.eng.tau.ac.il/\textasciitilde raja}}. 
 The results appear in Fig. \ref{fig:phase} in four plots:
 \begin{itemize}
 	\item[(a)]A phase transition diagram, following the methodology of \cite{phasedonoho}. We fix $m$ and variate $n$ and $k$ according to two auxiliary variables. We use a threshold of $\|\vx-\xh\|^2/\|\vx\|^2 \leq 10^{-4}$ and plot the resulting curve. $\mM$ is a normalized random Gaussian matrix, and $\vx$ is a sparse signal with the support selected uniformly at random with values drawn from the normal distribution. Results are averaged over 50 realizations. We compare to OMP, SP and BP. 
 	For BP \cite{basis} we used the CVX package \cite{cvx} to solve $\min \|\vx\|_1\ \mathrm{s.t}\ \vy = \mM\vx$, and then we improve its recovery using debiasing. 
 	\item[(b)] An experiment with noise, where we fix the size of the dictionary, corrupt the measurements by AWGN with $\sigma = 0.01\cdot\|\vy_0\|/\sqrt{m}$, where $\vy_0 = \mM\vx$, and plot the recovery error vs. the sparsity. Results are averaged over 1000 instances of $\mM$ and $\vx$, drawn as in (a). 
We compare to OMP, OMPR, CGIHT, SP, ALPS (ALgebraic PursuitS) \cite{Cevher11ALPS} and an oracle that has the true support $L$ of the $\vx$:
\begin{eqnarray}
\xh_{oracle} = \mM_L^\dag \vy.
\end{eqnarray}
For ALPS we use the 0-ALPS variant\footnote{We used the code from lions.epfl.ch/alps with the default parameters therein.}.  
We have compared also to BPDN, using the method in \cite{basis} but do not present it in the graph since it has performed worse than OMP.
\item[(c)] A test of the resilience to correlation in the dictionary. For a fixed $n,m,k$ we generate $\mM$ from a Gaussian distribution. We then increase the correlation in the dictionary by performing $\vd_{(i)}=\vd_{(i)} + \mu \vd_{(i+1)}$ for each atom, repeating the process five times to increase the effect. Results are averaged over 1000 realizations. We compare to OMP, OMPR, CGIHT, OLS, SP and ALPS. \rg{OMPR is very close to OMP for large coherence values.}
\item[(d)] Comparison between the number of iterations required to perform OLSR and IOLSR compared to OMP as a function of $k$ (same setup of experiment (b)). Recall that the computational cost of each iteration of our proposed methods is similar to OMP. 	
 \end{itemize}
 
 One may notice from the graphs the advantage that OLSR and IOLSR have on the other techniques both in the noisy and noiseless cases. In particular, notice the high coherence case, where the advantage of our strategies is more significant. Also note that OMPR is sensitive to high coherence in a similar way to OMP, where our approach shows better performance than all methods. IOLSR shows better performance than OLSR but requires a larger computational time. Yet, the number of iterations in both methods is relatively close to OMP. 

\FloatBarrier

\section{Theoretical performance guarantees of the synthesis least-residual techniques}\label{sec:theory_synthesis}
The OLSR and IOLSR performance guarantees are presented hereafter using a series of lemmas and theorems. To this end, the following \rg{notations are} used: $\sL$ is the true support s.t. $\vx(j)\neq 0$, $\forall j \in \sL$, and $0$ otherwise; $\sT$ is a set of $k$ indices that represent the chosen support of an algorithm at a specific iteration; $\Lt\triangleq \sL\setminus \sT$, $\ \Tt \triangleq \sT\setminus \sL$; $\ \xt(i)=\vx(i)\ \ \forall i\in \Lt$, and $0$ otherwise; $\yt= \mM\xt$; finally, $\ \kappa$ denotes the number of atoms in the true support that are not yet identified ($\kappa=|\Lt|=\|\xt\|_0$).
In the proofs here we rely on the following lemma from \cite{omp_rip_2}.  
 
\begin{lemma}	[Lemma 2.1 in \cite{omp_rip_2}] \label{lem:2_1}
	Given $\vx_1,\vx_2$, such that $\|\vx_1\|_0 + \|\vx_2\|_0 \leq 2k$, $\vx_1\perp \vx_2$, and a dictionary $\mM$ with a RIP constant $\delta<1$ of order $2k$, then: 
	\begin{equation}|\cos \angle (\mM\vx_1,\mM\vx_2)| \leq \delta. \label{eq:RLSOMP:noiseless:lemma_2_1} \end{equation} 
\end{lemma}
\begin{rmk}
By using the fact that  $\yt = \mM\xt$ and $P_\sT\yt = \mM\vx_0$ for some $\vx_0$ with a support $\sT$ (which is disjoint with the one of $\xt$), we have from Lemma~\ref{lem:2_1} that
		$$\|P_\sT \yt\|^2  = \|P_\sT\yt\|\|\yt\|\cos\angle(P_\sT\yt,\yt) \leq \|P_\sT\yt\|\|\yt\|\delta.$$
		Hence\vs
		\begin{equation}
		\|P_\sT\yt\|^2\leq \delta^2\|\yt\|^2.
		\label{eq:stepI:P}
		\end{equation}
		\rg{Because $\|P_\sT\yt\|^2+\|R_\sT\yt\|^2 = \|\yt\|^2$, we also have }
		\begin{equation}
		\|R_\sT\yt\|^2  \geq (1-\delta^2) \|\yt\|^2 \label{eq:stepI:R}.
		\end{equation}
 \end{rmk}
 
 Another lemma that we use is the following: 
 
\begin{lemma} \label{lem:c_delta}
 The square of the maximum cosine of the angle between $R_\sT\di$, where $i\in\Lt$, and $R_\sT\yt$ obeys 
 \begin{equation}
 \label{eq:stepIII:c_delta}
 \max_{i\in \Lt}\left( \cos \angle (R_\sT\di,R_\sT\yt)\right)^2 \geq {1\over \kappa}c_\delta,
 \end{equation}
 where $c_\delta = (1-\delta^2)(1-\delta)$. 
 \end{lemma}
 \begin{proof}
 	We prove by contradiction, inspired by the proof of Theorem 2.2 in \cite{omp_rip_2}. Consider $\|R_\sT\yt\|$: 
\begin{eqnarray}
\label{eq:39393939}
 \|R_\sT\yt\| =& {|\ip{R_\sT\yt,R_\sT\yt}|\over \|R_\sT\yt\|} = {|\ip{\sum_{i\in \Lt} x(i)R_\sT\di,R_\sT\yt}|\over \|R_\sT\yt\|}\\ \nonumber
    \leq & {\sum_{i\in \Lt} |\vx(i)\ip{R_\sT\di,R_\sT\yt}| \over \|R_\sT\yt\|}.
 \end{eqnarray}
 This leads to:\vspace{-0.2cm}
 \begin{equation} 
\label{eq:stepIII:up} 
 \|R_\sT\yt\| \overset{(a)}{\leq} \sum_{i\in \Lt} |\vx(i) \cos\alpha_i| \overset{(b)}{<}  \sqrt{{1\over \kappa}c_\delta} \|\xt\|_1\overset{(c)}{\leq} \sqrt{c_\delta}\|\xt\|,
 \end{equation} 
 where we define $\alpha_i = \angle(R_\sT\di,R_\sT\yt)$, and use the fact that $\|R_\sT\di\|^2\leq\|\di\|^2=1$ in transition $(a)$. For step $(b)$ we assume that for some positive constant $c_\delta$,  \  $|\cos\alpha_i|< \sqrt{c_\delta/\kappa}$, \  for all $i\in \Lt$. For $(c)$ we use the inequality $\|\xt\|_1\leq\sqrt{\kappa} \|\xt\|$.
 On the other hand, by \eqref{eq:stepI:R} and the RIP we have:
 \begin{align}\|R_\sT\yt\|\geq \sqrt{1-\delta^2}\|\yt\|\geq \sqrt{1-\delta^2}\sqrt{1-\delta}\|\xt\|	\label{eq:stepIII:dn}.
 \end{align}  
 
 Combining \eqref{eq:stepIII:up} and \eqref{eq:stepIII:dn}, we get that in order to produce a contradiction we \rg{can} set 
 $\sqrt{1-\delta^2}\sqrt{1-\delta} = \sqrt{c_\delta} $.
 \end{proof}
  While \rg{Lemma~\ref{lem:c_delta} may} be used to bound the contribution of a member of $\Lt$ from below, the following lemma gives an upper bound \rg{(in the noiseless case)} for the contribution to the estimation \rg{error} (defined in Lemma \ref{lem:intro:bckgrnd:coladderr}) of at least one of the erroneously selected atoms in the chosen support. 
  
  \rg{We will use this lemma, combined with Lemma~\ref{lem:c_delta}, to prove that all time OLSR or IOLSR did not converge to the true support, we may find a ``correct'' atom outside the estimated support that contributes more to the estimation error than a ``wrong'' atom contained in the current support. This will lead us to Theorem~\ref{thm:noiseless_delta} below  that shows that both techniques perfectly reconstruct the signal in the noiseless case (under some RIP conditions). An extension to the noisy case is provided afterwards in Theorem~\ref{thm:awgn}.} 
  
 \begin{lemma} \label{lem:min_RA_aj}
 	 Let $\vy = \mM\vx$. For a given support $\sT$ with $\zeta$ atoms that are not in the true support of $\vx$, there exists an atom $\dj$, $j\in\Tt$, such that: 
 	 \begin{equation}
 	 \label{eq:lem:min_RA_aj}
 	 {1\over \zeta}{\delta^2\over 1-\delta}\|\yt\|^2 \geq \ip{\vv{\Rj \dj},\vy}^2. 
 	 \end{equation}
 \end{lemma}

\begin{proof} 
Denote $\hat{\xt} = (\mA^T\mA)^{-1}\mA^T \yt$ (where $\mA = \mM_T$). We start with a bound on the norm of $P_\sT\yt$ from below as follows: 
 \begin{align*} 
 \|P_\sT\yt\|^2 &\overset{(a)}{\geq} (1-\delta) \|\hat{\xt}\|^2\\
 & \overset{(b)}{=} (1-\delta) \sum_{q\in \sT} {1\over \|\Rq\dq\|^4}\ip{\Rq \dq,\yt}^2 \\
 &\overset{(c)}{\geq} (1-\delta) \sum_{j\in \Tt} {1\over \|\Rj\dj\|^2}\ip{\Rj \dj,\vy}^2 \\
 &= (1-\delta) \sum_{j\in \Tt} \ip{\vv{\Rj \dj},\vy}^2,
 \end{align*}
 where $(a)$ follows from the RIP, $(b)$ follows from \eqref{eq:xh_val}, and $(c)$ is due to the fact that we sum on fewer atoms $|\Tt| < |T|$, reduce the power in the denominator from $4$ to $2$, and replace $\yt$ with $\vy$ since \rg{$\Rj \dj \perp \dq, \forall j \in \tilde{T}, q\in T \setminus \tilde{T}$}.
 The assertion in Lemma \ref{lem:min_RA_aj} now follows by noting that the smallest value of $\langle\vv{\Rj \dj},\vy\rangle^2$ for $j\in\Tt$ is not bigger than the average on $\zeta$ atoms in $\Tt$, and that $ \delta^2\|\yt\|^2\geq \|P_A\yt\|^2 $ from \eqref{eq:stepI:P}.  
\end{proof}
 Using the upper and lower bounds developed in the two previous lemmas, we can arrive at the following bound:   
  \begin{theorem}[RIP bound] \label{thm:noiseless_delta} 
  	Given $\vy=\mM\vx$, where $\|\vx\|_0=k$ and $\mM$ satisfies the RIP of order $2k$ with a constant $\delta\leq 0.445$, the OLSR and IOLSR algorithms yield perfect support reconstruction, $\sT=\sL$.
  \end{theorem}

We present now the proof for both OSLR and IOLSR. We consider the case of support-size based stopping condition for both techniques and error-based stopping condition for IOLSR (where $\epsilon_t =0$ since we analyze here the noiseless case).
  
   \begin{proof}[Proof of Theorem \ref{thm:noiseless_delta} ]
 	Assume that the algorithms converged to a support $\sT$, where $\Tt\neq \emptyset$.  	
 	Using Lemma \ref{lem:c_delta} we get that there exists $i\in \Lt\neq \emptyset$ such that: 
 	 	\begin{eqnarray}	 	 
 	 	\label{eq:stepV:up_1}
\nonumber	 \hspace{-0.3in}	\ip{\vv{R_\sT\di},\vy}^2    &\overset{(a)}{=} & \ip{\vv{R_\sT\di},R_\sT\yt}^2 \\   &    = &\|R_\sT\yt\|^2|\cos \alpha_i|^2  \geq \|R_\sT\yt\|^2{1\over \kappa}c_\delta,
 	 	\end{eqnarray}
 	 	where in transition $(a)$ we used $R_\sT \vy=R_\sT\yt$ that holds because $(\vy-\yt)\in\text{span}\{\mM_\sT\}$. 
 	 	
{\em OLSR case:} 	In this part of the proof, $T$ is the support at the stopping point of OLSR. At this stage, the support $T$ is of size $k$.
OLSR stops when there are no atoms that can replace the one that has been extracted (the condition for replacement is that the new residual is lower than the previous one). In \eqref{eq:stepV:up_1}, we have a lower bound on the error change (see Lemma~\ref{lem:intro:bckgrnd:coladderr}) at this stage. We will now turn to provide an upper bound for this error given that $\Tt \ne \emptyset$. Showing that the upper bound is smaller than the lower bound will contradict this assumption.

Denote by $j$ the index of the last atom extracted. Then before its extraction, the support was $T \cup \{j\}$. \rg{Since $|T| = k$, we have that $|T \setminus L| = |L \setminus T| = \kappa$ and thus $\zeta =|T \cup \{j\}\setminus L|\ge \kappa $.}
Combining the fact that the removal of the atom $j$ causes the smallest error change with the result of Lemma \ref{lem:min_RA_aj} (note that the rhs of \eqref{eq:lem:min_RA_aj} is the error change), leads to 
	 \begin{equation}
 	 \label{eq:min_RA_aj_OLSR_proof}
 	 \ip{\vv{R_\sT\dj},\vy}^2 
 	 \le  {1\over \kappa}{\delta^2\over 1-\delta}\|\yt\|^2. 
 	 \end{equation}
 	 
Combining \eqref{eq:stepI:R} and \eqref{eq:stepV:up_1} yields  $$\rg{\ip{\vv{R_\sT\dj},\vy}^2  \ge}  \|R_\sT\yt\|^2{1\over \kappa}c_\delta \overset{}{\geq} \|\yt\|^2{1\over \kappa}(1-\delta^2)c_\delta.$$ Therefore, to contradict the assumption that $\Tt \ne \emptyset$ (with which we started the proof), i.e., \rg{that the stopping criterion was reached with a reconstruction that includes a wrong atom in its support}, the following needs to hold:
 	$$ \|\yt\|^2{1\over \kappa}(1-\delta^2)c_\delta \geq {1\over \kappa}{\delta^2\over 1-\delta}\|\yt\|^2.$$
By eliminating the common terms ($\|\yt\|$ and $\kappa$),
substituting $c_\delta$ from Lemma~\ref{lem:c_delta} and reorganizing the terms, this inequality is equivalent to
$$ (1-\delta^2)^2(1-\delta)^2 \geq \delta^2.$$ 
It is easy to see that this inequality holds if $\delta \le 0.445$. 
 		
{\em IOLSR case:}	IOLSR differs from OLSR in that the new candidate is first inserted into the support, and then an atom for elimination is selected. (In OLSR the opposite happens: First an atom is removed, and then a new one that improves the residual is inserted if \rg{there is such a one). Thus, the proof is similar except for few changes and the bound is identical.}

Denote by $\sT$ the selected set at the beginning of the current iteration, by $i$ the atom added to $\sT$, by $j$ the candidate atom to be removed at this iteration, and $\sT' = \sT \cup \{i \}$. Assume that the current size of $\sT$ is $k$ and that $\tilde{T} \ne \emptyset$. Notice that we will continue removing $i$ and adding $j$ all the time that 
\begin{eqnarray}
\label{eq:IOLSR_error_based_stop_cond}
 \ip{\vv{R_\sT\di},\yt}^2 >  \ip{\vv{R_{\sT' \setminus \{j \} }\dj},\yt}^2, 
\end{eqnarray}
i.e., when the decrease in error due to adding $i$ to $T$ is greater than the increase in error due to the removal of $j$ from $\sT'$. 
\rg{Notice that once this condition does not hold, the algorithm enlarges the support by one. Thus, for a support-size based stopping criterion the algorithm will stop. For an error-based stopping condition the algorithm will continue unless $\tilde{T} = \emptyset$ and then IOLSR will stop as the error will be zero. Therefore, for proving that this method achieves perfect reconstruction for both stopping criteria, we show} now that all time that $\tilde{T} \ne \emptyset$, the inequality in  \eqref{eq:IOLSR_error_based_stop_cond} holds.

Setting $\yt' = \yt + \di\vx(i)$, leads to
\begin{eqnarray}
 \label{eq:ALSOMP:bla}
 		\|R_\sT\yt\|^2 &\geq \|R_\sT\yt\|^2 - \ip{\vv{R_\sT\di},\yt}^2\nonumber\\ &\overset{(a)}{=} \|R_{\sTt}\yt'\|^2
 \overset{(b)}{\geq} (1-\delta^2)\|\yt'\|^2,
\end{eqnarray}
where for transition $(a)$ we use \eqref{eq:coll_add:resid} and the fact that $R_{\sTt}\yt' = R_{\sTt}\yt$, and for $(b)$ we use \eqref{eq:stepI:R}.  		
Now, combining \eqref{eq:stepV:up_1} with \eqref{eq:ALSOMP:bla} leads to 
 \begin{eqnarray}
  \ip{\vv{R_\sT\di},\yt}^2 \ge {1-\delta^2\over \kappa}c_\delta \|\yt'\|^2.
 \end{eqnarray}

Applying the result of Lemma~\ref{lem:min_RA_aj} (with $T'$ and $\yt'$ instead of $T$ and $\yt$ respectively, \rg{where again we have $\zeta \ge \kappa$ as in the case of OSLR)}) provides us with 
 	 \begin{equation}
 	 {1\over \kappa}{\delta^2\over 1-\delta}\|\yt'\|^2 \geq \ip{\vv{R_{\sT' \setminus \{j \} }\dj},\yt}^2. 
 	 \end{equation}

Using simple arithmetic operations (similar to the OLSR case), we have that the condition $\delta \le 0.445$ is sufficient for \eqref{eq:IOLSR_error_based_stop_cond} to hold. Thus, it is guaranteed that if \eqref{eq:IOLSR_error_based_stop_cond} does not hold, then $\tilde{T} = \emptyset$ and we have converged to the correct support.
 \end{proof}

\begin{theorem} \label{thm:olsr:convergence}
	Let $\gamma = c_\delta -\delta^2/c_\delta$, with $c_\delta$ defined by \ \eqref{eq:stepIII:c_delta}. \rg{The number of iteration that noiseless OLSR needs for converge to a solution with a residual lower than $\epsilon_t$ is bounded by }
	\begin{equation}
	b = k \left(1+{1\over \gamma} \ln\left({\|\vy\|^2\over \epsilon_t e^{c_\delta}}\right)\right) \label{eq:olsr:conv_rate} .
	\end{equation} 
\end{theorem} 
Note that this is a worst case analysis. Empirically, only several iterations are required for convergence in addition to the OLS \rg{ ones (used to get an initial $k$-sparse solution for OLSR).} 

\begin{proof}
\rg{Notice that the first $k$ iterations of OLSR are actually the ones of OLS (which provides its initial estimate). Therefore, we denote by $\sT_n$ the estimated support of OLSR at its $n$th iteration, where $1 \le n \le k$ actually referes to the OLS iterations.
To analyze the convergence speed of OLSR, we first provide a bound on the error at the first $k$ iterations performed by OLS  and then compute the additional error decrease that happens at the subsequent iterations of OLSR. }

\rg{Let $\va$ be the atom added by OLS at iteration $n+1$. By recalling \eqref{eq:coll_add:resid} in Lemma \ref{lem:intro:bckgrnd:coladderr} and then applying Lemma \ref{lem:c_delta}, we have that the change in the residual error is given by 
\begin{eqnarray}
\|\Rn\vy\|^2-\|\Rnn\vy\|^2 = \ip{\vv{\Rn\va},\vy}^2 \geq \|\Rn\vy\|^2{c_\delta\over \kappa},
\end{eqnarray}
By reorganizing the terms, we have the following relationship between the residuals in subsequent iterations of OLS:
	$$\|\Rnn\vy\|^2\leq \left(1-{c_\delta\over\kappa}\right)\|\Rn\vy\|^2.$$
As we can replace $1/\kappa$ with $1/k$ to get a more restrictive bound and using the fact that $R_0 \vy = \vy$, we get 
	$$ \|R_{\sT_{k}}\vy\|^2 \le \left(1-{c_\delta/ k}\right)^k \|\vy\|^2. $$
For large enough $k$ this converges to 	$\|R_{\sT_{n=k}}\vy\|^2 \le e^{-c_\delta} \|\vy\|^2$. }

	Turning to the ``replacement part'' of the OLSR algorithm, assume that $tk$ ($t>0$) additional iterations were performed at this stage. \rg{From Lemma \ref{lem:min_RA_aj}, we have that at each step we take out of the support an atom with at most ${\delta^2 \over \kappa(1-\delta)} \|\yt\|^2$ contribution to the residual error. Using \eqref{eq:stepI:R} we bound this error increase  by
	${\delta^2 \over \kappa c_\delta} \|\Rn\yt\|^2$. On the other hand, from \eqref{eq:stepV:up_1}, we have that the atom added to the support reduces the error by  ${1 \over \kappa} c_\delta\|\Rn\yt\|^2$. Thus, we get that the error change obeys
\begin{eqnarray}
 \|\Rn\yt\|^2 - \|\Rnn\yt\|^2 \ge  {c_\delta - \delta^2/c_\delta \over k}  \|\Rn\yt\|^2.
\end{eqnarray}	
This leads to
	$\|\Rnn\yt\|^2 \le \left(1 -  {c_\delta - \delta^2/c_\delta \over k} \right) \|\Rn\yt\|^2$. Thus,}
	\begin{eqnarray}
	\|R_{\sT_{tk+k}} \vy \|^2 &\leq & \left(1-{\gamma\over k}\right)^{tk} \|R_{\sT_{k}}\vy\|^2  \\ \nonumber & \leq & \left(1-{\gamma\over k}\right)^{tk} e^{-c_\delta}\|\vy\|^2,
	\end{eqnarray}
where $\gamma =  c_\delta - \delta^2/c_\delta$ (note that $\delta < 0.445$ yields $\gamma>0$). Notice that for the estimation error to be less than $\epsilon_t$ we need $t \geq {\gamma^{-1}} \ln({\|\vy\|^2/ \epsilon_t e^{c_\delta}})$. Thus,  $(1+t)k$ is an upper bound on the number of iterations.
\end{proof}


We turn now to deal with the case that the measurements are corrupted by an additive white Gaussian noise (AWGN). To this end, we use the following notation $\vy_0 = \mM\vx$,\ \ $\vy=\vy_0+\vw$,\ \ $\vw\sim \NNN(0,\sigma^2\mI)$, where $\sigma^2$ is the noise variance.

\begin{theorem} \label{thm:awgn}
	Given measurements corrupted by AWGN with variance $\sigma^2$, and a parameter $a\geq 0$, perfect support reconstruction is achieved by OLSR and IOLSR with probability exceeding  $1-(\sqrt{\pi(1+a)\log n}n^a)^{-1}$ if 
		\begin{equation}
		{\|\vy_0\|^2\over \sigma^2k } \geq  { (1+\sqrt{1-\delta})^2(2(1+a)\log n) \over ((1-\delta)(1-\delta^2)  - \delta)^2}\nonumber .
		\end{equation}  
\end{theorem} 
\begin{proof}
The proof is very similar to the proof of Theorem~\ref{thm:noiseless_delta}. The only difference is that the lower and upper bounds used change due to the noise. Therefore, we describe here the changes that are needed to be done compared to the proof of Theorem~\ref{thm:noiseless_delta}.  Again we assume that $\tilde{T} \ne \emptyset$. 	We describe the proof for $\kappa=k$ \rg{(i.e., we have completely erred in the selection of the support)}, since this is the most restrictive case. The proof can be easily generalized and verified for the case $\kappa < k$. \rg{Thus, we assume} $\yt = \vy$. \rg{Throughout this proof, the following notation is used:}
	\begin{equation}
	\vw_\sT = {P_\sT\vw  \over \|\vy_0\|},\ \ w_i = {\ip{\vv{R_\sT\di},\vw}\over \|\vy_0\|}
	\label{eq:w_a_w_i}.
	\end{equation}
\rg{Consider the proof of Lemma \ref{lem:min_RA_aj}. Since we assume  $\yt = \vy$, then all the steps in this proof holds also for the case here except of the last step that bounds $\|P_\sT\vy\|$. Substituting it with}
	\begin{eqnarray}
	\|P_\sT\vy\| = \|P_\sT\vy_0 + P_\sT\vw\| &\leq & \|P_\sT\vy_0\|+\|P_\sT\vw\|  \\ \nonumber
	&\leq & (\delta +\|\vw_\sT\|)\|\vy_0\|,
	\end{eqnarray}	
	\rg{leads to the following 
	noisy form of the bound in Lemma \ref{lem:min_RA_aj}: }
	\begin{equation}
	{\delta +\|\vw_\sT\|\over \sqrt{k(1-\delta)}}\|\vy_0\| \geq \min_{j\in \sT}\left|\ip{\vv{R_\sT\dj},\vy}\right|
	\label{eq:4949494}.
	\end{equation}\vs

\rg{We now turn to bound $\left|\ip{\vv{R_\sT\di},\vy}\right|$ from below, where $i$ is an index of an atom in the true support (that has not been selected). First, notice that from the triangle inequality
\begin{eqnarray}
\left|\ip{\vv{R_\sT\di},\vy}\right| &=& \left|\ip{\vv{R_\sT\di},\vy_0 + \vw}\right| \\ &\ge &   \left|\ip{\vv{R_\sT\di},\vy_0 }\right| - \left|\ip{\vv{R_\sT\di}, \vw}\right|. \nonumber
\end{eqnarray}
By using \eqref{eq:w_a_w_i}  and the facts that $R_\sT$ is a projection ($R_\sT = R_\sT^2$, $R_\sT = R_\sT^T$) and $\|\vv{R_\sT\di}\| = 1$ followed by applying Lemma \ref{lem:c_delta} and \eqref{eq:stepI:R}, we get  }
	\begin{eqnarray}
	\label{eq:289292}
\nonumber	\left|\ip{\vv{R_\sT\di},\vy}\right| &\geq & \|R_\sT \vy_0\||\cos \alpha_i|  - \|\vy_0\||w_i| \\  & \hspace{-0.6in}\geq & \hspace{-0.4in} \|\vy_0\|\left(\sqrt{c_\delta (1-\delta^2)\over k} -|w_i|\right),
	\end{eqnarray}
Having these updated version of the bounds (of the lemmas) for the noisy case, we can repeat the same steps of Theorem~\ref{thm:noiseless_delta} for both \rg{OLSR and IOLSR}, where the only difference is that we plug these inequalities instead of the ones for the noiseless case. As in the proof of Theorem \ref{thm:noiseless_delta}, we aim at contradicting the assumption that the methods converged with $\di$ out of the estimated support. To this end, combine \eqref{eq:4949494} and \eqref{eq:289292} and replace $c_\delta$ with its value. \rg{This leads to the condition}
	\begin{equation}
	(1-\delta)(1-\delta^2)  - \delta \geq \|\vw_\sT\| + \sqrt{k(1-\delta)}|w_i|.
	\label{eq:AWGN:22332}
	\end{equation}
As a sanity check notice that if we set $\vw = 0$, we get back to the condition of the noiseless case.

Note that $w_i$ and $\vw_\sT$ reside in sub-spaces orthogonal to each other and that $\vw_\sT$ is with dimension $k$ whereas $w_i$ is of dimension 1.	Therefore, we can use the following confidence interval developed in \cite{DS_p_bigger_than_n}:
	\begin{eqnarray}
		\label{eq:DS:Pr}
&&   \hspace{-0.7in} 	Pr\left(\sup_{\rg{\dj, 1\le j \le n}}\left|\ip{\rg{\dj},\vw}\right| > \sigma \sqrt{2(1+a)\log n}  \right)  \\ \nonumber && \hspace{0.5in} \leq \left(\sqrt{\pi(1+a)\log n}n^a\right)^{-1}.
	\end{eqnarray}
Thus, with probability exceeding $1- \left(\sqrt{\pi(1+a)\log n}n^a\right)^{-1}$,
 we can bound the right hand side of \eqref{eq:AWGN:22332} by
	\begin{align}
	&\|\vw_\sT\| + \sqrt{k(1-\delta)}|w_i|  \leq{ \sqrt k (1+\sqrt{1-\delta}) \sigma \sqrt{2(1+a)\log n}\over \|\vy_0\|}.
	\label{eq:AWGN:22333}
	\end{align}
	Combining \eqref{eq:AWGN:22332} and \eqref{eq:AWGN:22333} concludes the proof. 
\end{proof}
The following is a consequence of Theorem \ref{thm:awgn}: 
\begin{corollary} \label{cor:1}
	If  $\sigma\leq c_1\|\vy_{0}\|/\sqrt k$ holds as in Theorem \ref{thm:awgn},  the error of the estimation $\xh$ of $\vx$ by OLSR and IOLSR is bounded with probability exceeding $1-\left(\sqrt{\pi(1+a)\log n}n^a\right)^{-1}$ by 
	\begin{equation}
	\|\vx-\xh\|\leq {\sqrt{2(1+a)\log n} \over 1-\delta}\sqrt k \sigma.
	\end{equation}
	
	.
\end{corollary}
\begin{proof}
	Assuming the conditions of Theorem \ref{thm:awgn} are satisfied, \rg{implies that we get} a perfect support reconstruction. Let $\mA = \mM_L$ be the vectors comprising the true support, then:
	\begin{align}
	\|\vx-\xh\| &= \|\vx - (\mA^T\mA)^{-1}\mA^T(\mA\vx+\vw)\| \nonumber\\
	&= \|(\mA^T\mA)^{-1}\mA\vw\|  \leq \|(\mA^T\mA)^{-1}\|\|\mA\vw\| \nonumber\\
	&\leq {1\over 1-\delta} \sigma\sqrt{2k(1+a)\log n} 
	\end{align}
	where in the first transition we write the expression for $\xh$ explicitly, the third uses a \rg{matrix-norm} inequality, and the fourth uses \eqref{eq:DS:Pr} and the bound $\min \text{eig} (\mA^T\mA) \geq 1-\delta$ due to the RIP condition on $\mM$.  
\end{proof}

Corollary \ref{cor:1} implies that under a perfect support reconstruction we have an error proportional to $\sqrt k \sigma$. \rg{Yet, if} the condition in Theorem \ref{thm:awgn} is violated (i.e., support recovery is not guaranteed) the error in the worst-case is still bounded by the following corollary:
\begin{corollary} \label{cor:2}
	The worst case error of the estimation $\xh$ of $\vx$ by OLSR and IOLSR is bounded with probability exceeding $1-\left(\sqrt{\pi(1+a)\log n}n^a\right)^{-1}$ by
	$$\|\vx-\xh\|\leq c_2 \sqrt{2k\sigma^2 (1+a)\log n},$$
	where $$c_2 = {(1+\delta^2)(1+\sqrt{1-\delta})\over (1-\delta)(1-\delta^2)-\delta}+{1\over 1-\delta}.$$
\end{corollary}

The proof of this corollary, analyzes the worst case scenario, where we have recovered erroneously all the atoms.

\begin{proof}
	We start by bounding the error norm as in Corollary \ref{cor:1}:
	\begin{align}
	\label{aeq:thth}
&	\|\vx-\xh\| = \|\vx - (\mA^T\mA)^{-1}\mA^T(\mA\vx+\vw)\| \nonumber\\
	&\leq \|\vx\|+\| (\mA^T\mA)^{-1}\mA^T\mA\vx\| + \|(\mA^T\mA)^{-1}\mA^T\vw\|\nonumber\\
	&\leq (1+\delta)\|\vy_0\| + \delta (1+\delta) \|\vy_0\| + {1\over 1-\delta} \sigma\sqrt{2k(1+a)\log n}. 
	\end{align}
	where in the third transition we used the RIP condition and the results of Corollary \ref{cor:1}, and assumed that the selected support is entirely erroneous; thus, \eqref{eq:stepI:P} holds. Now, to \rg{get Corollary \ref{cor:2}, compare \eqref{aeq:thth} and the bound in Corollary \ref{cor:1}. Notice that here (where $\sigma > c_1\|\vy_0\| /\sqrt k$) we have an additional factor of $(1+\delta)^2\|\vy_0\|$, i.e., we have a ``step'' at $\sigma = c_1\|\vy_0\| /\sqrt k$. Thus, to provide an upper bound that covers all values of $\sigma$, we need to calculate the straight line (in $\sigma$) that goes through the point 
	$$\left({c_1\|\vy_0\|\over\sqrt k},\ (1+\delta)^2\|\vy_0\| \right),$$
	and add it to the bound of Corollary \ref{cor:1},. }
	
\end{proof} 

Corollaries \ref{cor:1} and \ref{cor:2} imply that the IOLSR and OLSR \rg{errors are} proportional, up to $\OOO(\log n)$, to the error ($ \sqrt k \sigma$) of an oracle estimator that knows the true support. These results are similar to other near-oracle bounds developed for other methods including SP and CoSaMP \cite{BenHaim09Coherence, DS_p_bigger_than_n, RIP_oracle}.

\section{Efficient least-residual techniques for the analysis model} \label{sec:anal}
We turn now to extend the OLS approach to the analysis model, and improve it further by allowing backtracking. Notice that GAP, in a similar way to OMP, relies on correlations to select what rows from $\Omega$ to remove in its iterative process, although its objective is minimizing \eqref{eq:P1}. Thus, in a similar way to what we have done with FOLS and IOLSR, we propose here an efficient technique to re-calculate \eqref{eq:P1} when elements from the cosupport are added/removed. This allows us to propose GALS that removes columns from the cosupport directly using \eqref{eq:P1} and GALSR that, in a similar way to \rg{OLSR}, replaces elements in the cosupport based on the target objective \eqref{eq:P1}.

\subsection{Preliminaries for the analysis model algorithms}  
We start with some preliminary lemmas that will aid us in the derivation of two new algorithms for the analysis model.      
%
%
%
The first provides a variant of the restricted isometry property (RIP) for the analysis case \cite{giryes2014greedy}. We use it to ensure the existence of a solution to the analysis minimization problem in Definition~\ref{df:relaxedanal}. 
\begin{df}[$\Omega$-RIP] \label{def:O-RIP}  
	A matrix $\mM$ is said to satisfy the \ORIP of order $s$ with constant $\delta^\Omega_s$, if for any $\vx$ such that $\mO\vx$ has more than $s$ zeros (i.e. $\vx$ is orthogonal to at-least $s$ rows from $\mO$) 
	$$ (1-\delta^\Omega_s)\norms{\vx}\leq \norms{\mM\vx}\leq(1+\delta^\Omega_s)\norms{\vx}. $$
\end{df}
A large range of sampling operators satisfy the \ORIP with very high probability, e.g., random Gaussian or sampled Fourier matrices with similar probabilities to the regular RIP \cite{giryes2014greedy}. 
The definition of the \ORIP leads us to the following lemma.
 
\begin{lemma} \label{lem:B_S} 
Given that $\delta_s^\Omega<1$, the matrix 
\begin{eqnarray}
\label{eq:B_lambda}
\mB_\sLam = \mM^T\mM+\mO_\sLam^T\mO_\sLam,
\end{eqnarray}
is invertible whenever $|\sLam|\geq s$.
\end{lemma}    
\begin{proof} 
Assume by contradiction that $\exists \vx\neq 0$ such that
	$$\vx^T\mB_\sLam\vx = \norms{\left[\begin{array}{c}\mM\\ \mO_S \end{array}\right]\vx    }=0.$$
This implies that both $\mM \vx = 0$ and $\mO_S=0$. Yet, according to the \ORIP,  if $\mO_\sLam\vx=0$ and $\delta_s^\Omega<1$ then $\mM \vx \neq 0$. 	
\end{proof}

Notice that here and below, we just use the lower inequality in the RIP condition, which is basically equivalent to requiring that $\mM$ does not have $s$-cosparse vectors in its null-space. This is a very mild assumption that we take in this work as it is clear that if $\mM\vx =0$ it is impossible to recover $\vx$. 

\begin{lemma} \label{lem:P1_sol} 
If $\delta^\Omega_{|\sLam|}<1$, then the solution to \eqref{eq:P1}  satisfies 
	\begin{equation} \label{eq:P1_sol}
	 \MMM{cc}{\mB_\sLam & \mM^T \\ \mM & \Zer}  \MMM{c}{\xh\\ \vz} = \MMM{c}{\mM^T \vy \\ \vy },
	\end{equation} 	
	where $\mB_\sLam$ is defined in \eqref{eq:B_lambda} in Lemma \ref{lem:B_S}.
\end{lemma}
\begin{proof}
	Let $\xh$ be a solution to $\mM\xh=\vy$ satisfying \eqref{eq:P1_sol}. We will demonstrate that $\xh$ has the minimal $\norms{\mO_{\sLam}\xh}$, implying that it is the solution to \eqref{eq:P1}. To this end, let $\xh_2\neq\xh$ be another possible solution (i.e. $\mM\xh_2=\vy$). Then, by substituting $\xh_2=\xh_2-\xh+\xh$, we have
	\begin{eqnarray}
	\label{eq:mo_Lam_xh2_ineq}
\norms{\mO_{\sLam}\xh_2}=\norms{\mO(\xh_2-\xh)}
+\norms{\mO\xh} \\ \nonumber +2(\xh_2-\xh)^T\mO_\sLam^T\mO_\sLam\xh.
	\end{eqnarray}
By the assumption that $(\xh_2-\xh)^T\mM^T=0$ and using the relationship defined in \eqref{eq:P1_sol}, we have that $2(\xh_2-\xh)^T\mO_\sLam^T\mO_\sLam\xh = 2(\xh_2-\xh)^T\mB_\sLam\xh = 2(\xh_2-\xh)^T\mM^T(\vy-\vz) =0$. Combining this with \eqref{eq:mo_Lam_xh2_ineq}, we have that 
 	\begin{align}
	\norms{\mO_{\sLam}\xh_2}&= \norms{\mO(\xh_2-\xh)}+\norms{\mO\xh} > \norms{\mO\xh},
	\end{align}
To justify the strict inequality we need to prove that $\mO(\xh_2-\xh) \ne 0$. If this does not hold then $\mO_\sLam(\xh_2-\xh) =0$ for any $\sLam$. Thus, from the RIP condition, we have $\mM(\xh_2-\xh) \ne 0$, which contradicts the fact that $\mM\xh_2$ also satisfies $\mM\xh_2 = \vy$. 
This completes the proof.
\end{proof}
\begin{lemma} \label{lem:P1_xh}
	Let $\delta^\Omega_{|\sLam|}<1$ and define $\CL=(\mM\BL\mM^T)^{-1}$. Then the value for $\xh$ in \eqref{eq:P1} can be written as 
	\begin{equation}
	\label{eq:axh}
	\xh =  \BL \mM^T \CL \vy.
	\end{equation} 
	
\end{lemma} 
\begin{proof} 
	We start by inverting the matrix in Lemma \ref{lem:P1_sol}. From the block version of the matrix inversion lemma, we have
	\begin{align}
\label{eq:B_lambda_MTM_block_inverse}
	&\MMM{cc}{\mB_\sLam & \mM^T \\ \mM & \Zer}^{-1}  \\ \nonumber & \hspace{0.3in} = \MMM{cc}{ \BL - \BL \mM^T \CL \mM \BL  & \BL \mM^T \CL  \\ \CL \mM \BL  & -\CL} .
	\end{align}
Plugging \eqref{eq:B_lambda_MTM_block_inverse} into \eqref{eq:P1_sol} leads to 
	$$\xh =  (\BL - \BL \mM^T \CL \mM \BL)\mM^T\vy+  \BL \mM^T \CL \vy.  $$
Noticing that the first term is equal to zero, which yields the desired outcome. 
\end{proof}

In our proposed algorithms we rely on a partition of $\RR^n$ into two parts using the following orthogonal matrix
\begin{equation}\label{aeq:Q}
\mQ = \MMM{cc}{\QM & \QO} \in \RR^{n\times n},
\end{equation} 
where $\QM \in \RR^{n\times m}$ is an orthogonal basis for the row span of $\mM$, and $\QO\in \RR^{n \times (n-m)}$ spans the subspace orthogonal to $\mQ$ (i.e., spans the null-space of $\mM$). The matrix $\QM$ is calculated using the SVD decomposition of $\mM$. The definition of $\mQ$ leads to the following lemma that relates it to the projection onto $\mO$.       
\begin{lemma}
	\label{alem:MSE}
	Let $\delta^\Omega_{|\sLam|}<1$ and define $\LL = \QO^T\OL^T\in \RR^{(n-m)\times |\sLam|}$.\footnote{\rg{Following our notation, in the analysis case $\sLam$ subscript denotes a subset of rows. Yet, $\LL$ is an exception; in this case it denotes a subset of columns.}} Then the minimzer of \eqref{eq:P1} satisfies 
\begin{eqnarray}
\label{eq:OL_xh_residual_norm}
\norms{\OL\xh}  = \vy^T(\CL-\mI)\vy = \norms{R_{\LL} \OL \mM^\dagger \vy },
\end{eqnarray} 
where $R_{\LL} = \mI - P_{\LL}$.
\end{lemma}
\begin{proof} 
	According to Lemma \ref{lem:P1_xh} we have 
	$$\norms{\OL\xh} = \vy^T \CL \mM \BL  \OL^T\OL \BL \mM^T \CL \vy   .$$
	Plugging $\OL^T\OL = \mB_\sLam - \mM^T\mM$ (see \eqref{eq:B_lambda}) into the above expression, we get  
\begin{eqnarray}
\label{eq:OL_xh_eq_vyT_CK_ml_vy}
\norms{\OL\xh}  = \vy^T(\CL-\mI)\vy.
\end{eqnarray}	
	This is the first equality in the Lemma. The second equality is proven in Appendix~\ref{sec:eq:OL_xh_residual_norm_proof}. 
\end{proof}  
\begin{lemma} \label{alem:errinc}  
	Let $\sLam$ be a set of indices, and $i\notin \sLam$ an index of a row in $\mO$. Let $\xh_1$ be the solution to \eqref{eq:P1} with $\sLam$ and $\xh_2$ be the solution to \eqref{eq:P1} with $\sLami$. The change in error is then 
\begin{eqnarray}
\label{eq:analysis_atom_addition}
\norms{\OLi \xh_2} - \norms{\OL\xh_1}   = {(\beta -\gamma )^2\over 1+\alpha},
\end{eqnarray} 
	where $\alpha = \vq^T (\LL\LL^T)^{-1}\vq$, $\vq = \QO^T\oi$, $\beta = \oi^T\mM^\dagger\vy_0$, $\gamma = \vq^T(\LL^\dagger)^T\OL\mM^\dagger \vy_0$.
\end{lemma} 
\rg{In the lemma $\oi$ is the $i$th row of $\mO$.}
Its proof  appears in Appendix~\ref{sec:alem:errinc_proof}. 
\rg{In the noiseless case ($\vy = \vy_0$), we have that $\beta = \vb_{(i)}$ and $\gamma = \vq^T\vgam = \vgam_{(i)}$ and $\alpha =\valpha_{(i)}$, where $\vb,\vgam$ and $\valpha$ are auxiliary vectors used in algorithms~\ref{tb:GALS} and \ref{tb:GALSR}.}

\begin{lemma} \label{alem:errexc}  
	Set $\sLam$, $\alpha$, $\beta$, $\gamma$ as in Lemma \ref{alem:errinc}, and let $i\in \sLam$ be an index of a row in $\mO$. Let $\xh_1$ be the solution to \eqref{eq:P1} with $\sLam$ and $\xh_2$ be the solution to \eqref{eq:P1} with $\sLammi$. The change in error is then 
\begin{eqnarray}
\label{eq:analysis_atom_deletion}
\norms{\OL \xh_1} - \norms{\OLmi\xh_2}  = {(\beta -\gamma )^2\over 1-\alpha}.
\end{eqnarray}
\end{lemma} 
The proof is similar to the proof of Lemma \ref{alem:errinc}.

\subsection{Proposed analysis greedy least squares based  methods} 
\subsubsection{Greedy analysis least squares (GALS)}
GALS operates similarly to GAP with the objective function \eqref{eq:P1} in Definition \ref{df:relaxedanal}.  Staying true to the OLS selection criteria, the atom to be deleted from $\sLam$ at each iteration is the one that will lower the residual the most (we refer to $\norms{\OL\vx}$ as the residual in this context). The selection rule is as defined by Lemma \ref{alem:errexc}, where at each iteration we seek the entry corresponding to an atom in the current cosupport whose entry is maximal. \rg{The stopping criterion is either a target error or a pre-defined cosupport size.
Notice that the GALS technique only removes atoms from the cosupport as the GAP strategy does, but with a least squares criterion. Therefore, as GAP is the analysis version of OMP, GALS may be viewed as the analysis version of OLS. Thus, we would expect to get better performance with it compared to GAP.} 

\rg{The pseudo-code of an efficient implementation of GALS is described in Algorithm \ref{tb:GALS}.
This algorithm starts with a full cosupport and then iteratively removes atoms from it. To find the row to be removed efficiently, our strategy uses the criterion in Lemma~\ref{alem:errexc}. Then it updates the residual efficiently also using the auxiliary variables. }
We postpone the discussion related to the update rules of the different auxiliary variables required by GALS to Section \ref{asec:updag}.
 
\begin{table}[t]	
	\caption{Greedy analysis least squares (GALS)}	
	\label{tb:GALS}	
	\centering	
	\begin{tabular}{l}	
		\toprule	
		Input: operator $\mO$, measurement $\vy$, sensing matrix $\mM$,\\ ~~~~~~~~~~ either cosparsity $l$ or target residual $\epsilon_t$ 	\\
		Output:  $\xh$, and $\sLam$ its cosupport with either $|
		\sLam|=l$ or $\norms{\OL\xh}\leq \epsilon_t$ 	\\
		\midrule
		\textbf{init:}\\	
		\quad $\sLam\leftarrow \{1,\dots,p\}$, $\vb \leftarrow \mO\mM^\dagger \vy$ \\ 
		\quad calculate $\mQ_{\mM^\perp}$ according to \eqref{aeq:Q} \\
%
\quad$\mL \leftarrow \mQ_{\mM^\perp}^T\mO^T$  \quad\quad \textbf{-Auxiliary} \\  
\quad $\mG \leftarrow (\mL\mL^T)^{-1} $, $\vgam \leftarrow \mL^T\mG\mL\vb$, $\valpha \leftarrow \text{diag}(\mL^T\mG\mL)$  \quad \textbf{-Auxiliary} \\
		\quad $\epsilon_0 \leftarrow \norms{\vb} -\norms{\vgam}$   \quad\quad \textbf{-Initial residual} \\
		\textbf{while} $|\sLam|>l$ \textbf{or} $\epsilon_0 > \epsilon_t$ \quad \textbf{-Only one condition is tested}  	\\
		\quad$i \leftarrow \arg \max_{i\in \sLam} |\vb_{(i)} - \vgam_{(i)}|^2 / (1-\valpha_{(i)} )$(***) \textbf{-Find row to remove} 	\\
		\quad$\sLam \leftarrow \sLam\setminus i$ \quad\quad \textbf{-Remove row from cosupport}	\\
		\quad $\{\valpha,\vgam, \mG\}\leftarrow \mathrm{updRemA}(\mO,\mQ_{\mM^\perp},\vb,\mG,\vgam,\valpha,i)$\quad (Alg. \ref{tb:aupd})  \\ 

		\quad $\epsilon_0 \leftarrow \norms{\vb_\sLam} -\norms{\vgam_\sLam}$   \quad \quad\textbf{-Update residual}\\
		\textbf{end while}	\\
		\textbf{return} $\sLam$,  $\xh = \min_{\xh} \norms{\OL{\xh}} \text{ s.t. } \vy = \mM\xh $	\\
		\bottomrule	
	\end{tabular}	
\end{table}	
\subsubsection{Greedy analysis least squares with replacement (GALSR)}
As the analysis-model equivalent to OLSR, we now propose a novel algorithm that is based on GALS that we name
GALS with replacement (GALSR). 
 As a first step it uses GALS to produce a cosupport estimation of size $l$. \rg{Then it iteratively reduces the error by replacing atoms in the support. It first finds the atom that improves the  error the most and adds it, and then it removes the atom that least contributes to the error. The algorithm halts if the error stops decreasing.  }
 
\rg{Our efficient implementation of GALSR is described in Algorithm \ref{tb:GALSR}. 
The atoms addition and removal are performed  according to Lemmas \ref{alem:errinc} and \ref{alem:errexc}, respectively, until convergence occurs. At each of the GALSR iterations, it adds an atom from the complement of the currently selected cosupport by seeking the minimizer of \eqref{eq:analysis_atom_addition} in Lemma \ref{alem:errinc}. Then, the atom that will lower the residual the most as defined in \eqref{eq:analysis_atom_deletion} in Lemma \ref{alem:errexc} is excluded from the cosupport. The algorithm stops when the error stops decreasing. It outputs a reconstruction with a cosupport of a pre-defined size $l$ that is given as an input  to the algorithm. }
\begin{table}[t]	
	\caption{Greedy analysis least squares with replacement (GALSR)}	
	\label{tb:GALSR}	
	\centering	
	\begin{tabular}{l}	
	\toprule	
	Input: operator $\mO$, measurement $\vy$, sensing matrix $\mM$, cosparsity $l$ \\
	Output:  $\xh$, and $\sLam$ its cosupport with $|
	\sLam|=l$	\\
	\midrule
	\textbf{init:}\\	
	\quad $\{\sLam,\valpha,\vgam,\mG,\vb,\epsilon_0\}\leftarrow$GALS($\Omega$,$\vy$,$\mM$,$l$)\\
	\textbf{loop} 	\\
	\quad$j \leftarrow \arg \min_{j\notin \sLam} |\vb_{(i)} - \vgam_{(i)}|^2 / (1+\valpha_{(i)} )$ (**) \textbf{Find row to add}	\\
	\quad$\sLam \leftarrow \sLam+ j$ \quad \textbf{Add row to the cosupport}	\\
		\quad $\{\valpha,\vgam, \mG\}\leftarrow \mathrm{updAddA}(\mO,\mQ_{\mM^\perp},\vb,\mG,\vgam,\valpha,j)$\quad (Alg.  \ref{tb:aupd}) \\ 
	\quad$i \leftarrow \arg \max_{i\in \sLam} |\vb_{(i)} - \vgam_{(i)}|^2 / (1-\valpha_{(i)} )$	 (***) \textbf{-Find row to remove}\\
	\quad$\sLam \leftarrow \sLam\setminus i$   \quad \textbf{Remove row from the cosupport}	\\
		\quad $\{\valpha,\vgam, \mG\}\leftarrow \mathrm{updRemA}(\mO,\mQ_{\mM^\perp},\vb,\mG,\vgam,\valpha,i)$\quad (Alg.  \ref{tb:aupd}) \\ 
	\quad \textbf{exit loop if } $\epsilon_0 \leq  \norms{\vb_\sLam} -\norms{\vgam_\sLam}$   \textbf{-Check if residual decreased}  \\
	\quad $\epsilon_0 \leftarrow \norms{\vb_\sLam} -\norms{\vgam_\sLam}$ \quad \quad \textbf{-Update residual} \\
	\textbf{end loop}	\\
	\textbf{return} $\sLam$,  $\xh = \min_{\xh} \norms{\OL{\xh}} \text{ s.t. } \vy = \mM\xh $	\\
	\bottomrule	
\end{tabular}	
\end{table}	
\subsubsection{Update routines for the GALS and GALSR auxiliary variables} \label{asec:updag}
Unlike the synthesis case where an addition of two length $n$ vectors was sufficient for FOLS and OLSR, in the analysis case we need to keep more data in memory. 
We use the following auxiliary variables in their calculation: $\mQ_{\mM^\perp}\in\RR^{n\times (n-m)}$, $\mL_\sLam$ and  $\mG$, where $\sLam$ is a given cosupport.
 \begin{itemize}
 \item  $\mQ_{\mM^\perp}\in\RR^{n\times (n-m)}$  is the orthogonal complement to the range of $\mM$. Refer to \eqref{aeq:Q} for its mathematical formulation. It can be calculated by performing a QR decomposition or SVD of $\mM$, among else.
 \item $\mL_\sLam= \mQ_{\mM^\perp}^T\mO^T_\sLam\in \RR^{(n-m)\times p}$ is as in Lemma~\ref{alem:MSE}. $\mL_{\sLam}$ is used to compute $\mG$ (defined next).
\rg{It needs to be calculated explicitly only in the GALS initialization phase, where $ \mL_\sLam = \mL $ since $\sLam =  \{1,\dots,p\}$ and $\OL = \mO$.  Note that in the rest of the algorithm we do not  calculate $\mL_\sLam$ directly. }
 \item $\mG = (\LL\LL^T)^{-1}\in \RR^{(n-m)\times (n-m)}$  is \rg{updated in Algorithm \ref{tb:aupd} by the matrix inversion lemma when $\sLam$ changes.} 
 \end{itemize}

\rg{Notice that we also use the routines ``updAddA'' and ``updRemA'' in Algorithm~\ref{tb:aupd} that update the auxiliary variables used in  the GALS and GALSR methods upon atom addition and deletion. These procedures use three additional auxiliary variables $\vb,\vgam$ and $\valpha$, which are used for calculating efficiently the updated error (as in Lemma \ref{alem:errinc} and \ref{alem:errexc}).}
\begin{table*}[t]		
	\caption{Update procedures for the auxiliary variables in GALS and GALSR}		                     
	\label{tb:aupd}		
	\centering		
	\begin{tabular}{ll}		
		\toprule
	\begin{minipage}{0.4\textwidth}
			$\mathrm{updAddA}(\mO,\mQ_{\mM^\perp},\vb,\mG,\vgam,\valpha,j)$
			\begin{addmargin*}{0.3cm}
				$\v \leftarrow \mQ_{\mM^\perp}\mG\mQ_{\mM^\perp}^T \oj$\\
				$\mG \leftarrow \mG - {1\over 1+\valpha(j) }\v\v^T$\\
				$\vgam \leftarrow \vgam + {\vb(j) - \vgam(j)\over 1+\valpha(j) }\mO\v$\\  
				$\valpha \leftarrow \valpha - {1\over 1+\valpha(j) }\mO\v\odot\mO\v$  \\
				\textbf{return} $\valpha,\vgam,\mG$
			\end{addmargin*}

	\end{minipage}
	& 
	\begin{minipage}{0.4\textwidth}
	$\mathrm{updRemA}(\mO,\mQ_{\mM^\perp},\vb,\mG,\vgam,\valpha,i)$
				\begin{addmargin*}{0.3cm}
	$\v \leftarrow \mQ_{\mM^\perp}\mG\mQ_{\mM^\perp}^T \oi$\\
	$\mG \leftarrow \mG + {1\over 1-\valpha(i) }\v\v^T$\\
	$\vgam \leftarrow \vgam - {\vb(i) - \vgam(i)\over 1-\valpha(i) }\mO\v$\\  
	$\valpha \leftarrow \valpha + {1\over 1-\valpha(i) }\mO\v\odot\mO\v$ \\
	\textbf{return} $\valpha,\vgam,\mG$
				\end{addmargin*}
\end{minipage}
	\\
		\bottomrule		
		\tiny	*\quad '$\odot$' designates element-wise multiplicaton 
	\end{tabular}		
\end{table*}	
\subsubsection{Properties and complexity of GALS and GALSR} 
\label{asec:complexity}
Once the initialization phase is finished, the most computationally expensive stage is the calculation of $\mO\v$ in Algorithm \ref{tb:aupd}, which costs $pn$ flops. This is in par with the complexity required by GAP, but the algorithms we propose alleviate the need to calculate $\xh$ at each stage explicitly, which might save some running time. On the initialization stage, the inversion of a $n-m\times n-m $ size matrix to create $\mG$ is the most expensive stage, comparable to the complexity of calculating the initial least squares estimate in GAP. The second time consuming step is the calculation of $\mQ_{\mM^\perp}$ that can be calculated using the QR decomposition of $\mM$. This has a complexity of $\OOO(mn^2)$. In case that the same $\mM$ and $\mO$ are used for more than one measurement, these calculations can be done only once. This may save a considerable execution time. 

\subsubsection{Accelerating computation for special cases}
In most scenarios where the dimension of the data is large, there are fast ways to calculate $\mM\vx$ and $\mO\vx$ instead of expensive matrix vector multiplication. Several well known examples where such "fast-multiply" exists are the Fourier, Haar, wavelet, and 2D difference transforms, making their use as analysis and or sampling operators appealing for high-dimensional data. Efficient multiplication schemes can be used to accelerate the calculation of matrix inverses (where the matrix is defined as the successive application of an operator and its adjoint). In this section, we demonstrate one such acceleration for the case where $\mM$ is a sub-sampled 2D Fourier transform and $\mO$ is the 2D \emph{circular} difference operator. To this end denote by $\sLam_{\mM}$ the indices of the 2D-FFT used in the sampling, i.e. $\mM\vx = \mathrm{FFT}_{\sLam_{\mM}}(x)$. In this case (and other cases where $\mM$ is orthogonal), calculating $\mQ_{\mM^\perp}\vx$ amounts to calculating the $\mathrm{FFT}$ of $\vx$ and taking the complement of $\sLam_{\mM}$. 

The acceleration we propose in this case is as follows. Let $\vx$ be a $s\times s$ image, $n=s^2$, $m=|\sLam_{\mM}|$, and $\mF$ be the 1D-DFT matrix of size $s\times s$. Then $\mathrm{FFT}(x) = (\mF\otimes\mF)\vx$ where $\otimes$ denotes the Kronecker product. For $\mD$, the 1D circular difference operator, $\mD(i,i-1:i) = [1,-1],\ \forall i\in\{1\dots s-1\}$, $\mD(0,0) = -1$, $\mD(0,s-1)=1$. Thus, 
\begin{equation}
\label{eq:5678} 
\mM = (\mF\otimes\mF)_{\sLam_{\mM}}\ ; \ \ \ \mO = \left[ \begin{array}{c}\mD \otimes \mI \\  \mI \otimes \mD \end{array} \right].
\end{equation} 
The calculation of the first operation in both $\mathrm{updAddA}$ and $\mathrm{updRemA}$, which is the most time-consuming one, becomes: 
				$$\mQ_{\mM^\perp}\mG\mQ_{\mM^\perp}^T \oj \ \ \Rightarrow \ \ \mathrm{FFT}_{\sLam_{\mM}^c}(\mG\cdot \mathrm{IFFT}_{\sLam_{\mM}^c}(\vw_{(j)})),$$
which can be calculated in a fast manner. 
We remain with the need to calculate $\mG\v$ for some vector $\v$. To this end, note that the rows of the Fourier matrix are the eigen-vectors of the second derivative operator $\mD^T\mD$. More specifically, the matrix $\matrx{T}=\mF \mD^T\mD\mF^T$ is  diagonal with $\text{diag}(\matrx{T})=-4\sin({\pi\over s}0:s-1 )$. 
Now, by recalling that
$\mG_\sLam^{-1} = \mL_\sLam\mL_\sLam^T=\mQ_{\mM^\perp}^T\mO_\sLam^T\mO_\sLam\mQ_{\mM^\perp}$, for $\mG_0 = \mG_{\sLam = \{1..p\}}$, we get: 
\begin{align*}
\mG^{-1}_0 &=  \mI_{\sLam_{\mM}^c}\mF\otimes\mF(\mD^T\mD\otimes\mI + \mI \otimes \mD^T\mD ) \mF^T \otimes \mF^T  \mI_{\sLam_{\mM}^c}^T\\
&=  \mI_{\sLam_{\mM}^c} (\mF\mD^T\mD\mF^T\otimes \mI+ \mI\otimes\mF\mD^T\mD\mF^T)  \mI_{\sLam_{\mM}^c}^T\\
&=\mI_{\sLam_{\mM}^c} (\matrx{T}\otimes\mI+\mI\otimes\matrx{T})\mI_{\sLam_{\mM}^c}^T
\end{align*}
which is also diagonal. Denoting $\vec{t} = \text{diag}(\mG_0)$, using $\mO_\sLam^T\mO_\sLam = \mO^T\mO-\mO_{\sLam^c}^T\mO_{\sLam^c}$ and applying the Woodbury matrix identity, leads to the following relationship for some vector $\v$:
$$ \mG_\sLam\v = \vec{t}\odot\mathrm{FFT}_{\sLam_{\mM}^c}\left({\mO_{\sLam^c}^T\tilde\mG_{\sLam^c}}\mO_{\sLam^c}\cdot\mathrm{IFFT}_{\sLam_{\mM}^c}(\vec{t}\odot\v)\right)+\vec{t}\odot\v, $$
where $\tilde\mG_{\sLam^c}$ is a $k\times k$ matrix.  
This results in a much cheaper operation as multiplying $\mO_{\sLam^c}$ is only $k\times n$ flops.

\subsubsection{Extension to noisy measurements}
In the preceding sections we have introduced a new approach to the analysis model signal reconstruction based on the noiseless objective function in Definition \ref{df:relaxedanal}. In many practical applications, a noise is present and taking the noise into consideration in the program is paramount for practical applications. To this end, GAP with noise (GAPn) has been introduced in \cite{nam2011recovery}. This technique replaces the $\vy= \mM\xh$ constraint in \eqref{eq:P1} in Definition~\ref{df:relaxedanal} with  $\norms{\vy-\mM\xh}\leq \epsilon_{\vw}$, where $\epsilon_{\vw}$ is the energy of the noise vector $\vw$ in \eqref{eq:first}. 
The value $\epsilon_{\vw}$ in GAPn should be set depending on the problem at hand. During the experiments, we saw ambiguous results where in some cases setting $\epsilon_{\vw}$ proportional to the noise power has been the option with the better reconstruction accuracy, while in other cases setting it to a small constant has led to a better results. 
         
We use a similar approach to GAPn to extend our strategies to the noisy case. To retain a reasonable run time we propose to use an alternating minimization scheme, where we introduce a vector $\wh$ that is calculated in the following way. Let $\sLam$ be the selected cosupport. Then  
\begin{equation} \label{aeq:nw}
\wh = \argmin_{\|\vw\| = \epsilon_{\vw}} \norms{R_{\LL} \OL \mM^\dagger (\vy-\vw)},
\end{equation} 
where we get \eqref{aeq:nw} by extending \eqref{eq:OL_xh_residual_norm} in Lemma \ref{alem:MSE} to the noisy case (adding the perturbation $\vw$ that represents the noise). 

To use the proposed scheme in GALS and GLASR we replace the rows marked with either (**) or (***) in algorithms \ref{tb:GALS} and \ref{tb:GALSR} with the procedure in Algorithm \ref{tb:aupdn}, and use the augmented vectors $\tilde{\vgam}$ and $\tilde{\vb}$ in the selection step (denoted by either `$\arg\min$' or `$\arg\max$'). 

There are several ways to solve the minimization problem in \eqref{aeq:nw}, which is part of Algorithm \ref{tb:aupdn}. We provide three methods hereafter and leave the choice of which to choose to the reader. 

The first option is the strategy Nam et al. take in \cite{nam2011recovery}.  They proposed to relax their equivalent of \eqref{aeq:nw} with  $$\min_{\vw} \norms{R_{\LL} \OL \mM^\dagger (\vy-\vw)} + \lambda\norms{\vw},$$
	where a search for a suitable $\lambda$ is performed in each iteration.  
	
	 The second approach is to replace $ \norms{R_{\LL} \OL \mM^\dagger (\vy-\vw)} $ in \eqref{aeq:nw} with its equivalent from Lemma \ref{alem:MSE}, $ (\vy-\vw)^T(\mC-\mI)(\vy-\vw) $. This leads to the following minimization problem 
\begin{eqnarray}
\label{eq:wh_minmization_vw_epsilon}
\wh = \argmin_{\|\vw\| = \epsilon_{\vw}} (\vy-\vw)^T(\mC-\mI)(\vy-\vw). 
\end{eqnarray}
Using Lagrange multipliers yields 
	 $$\wh = (\mC - (1-\lambda)\mI)^{-1}(\mC-\mI)\vy, $$
	 where $\lambda$ should be determined such that $\|\wh\|=\epsilon_{\vw}$.
	 Storing the $m$ eigenvalues and eigenvectors of $\mC-\mI$ in memory and updating them when an atom is added or removed from $\sLam$, mitigates the need to solve multiple least-squares problems each time at the cost of tracking the eigenspace of $\mC-\mI$ under rank-1 updates (e.g. by using one of the methods in \cite{mitz2017symmetric} and references therein).
	 
The third applicable approach is finding the value of $\lambda$ that correponds to solving \eqref{eq:wh_minmization_vw_epsilon} analytically \cite{gander1989constrained}. 
	 This can be done by determining the minimal eigenvalue of the following quadratic eigenvalue problem 
	 \begin{align*}
	 \hat{\lambda} = \arg\min_\lambda \left\{ \mH^2 - {1\over \epsilon_{\vw}^2}\mH\vy\vy^T\mH -2\lambda \mH +\lambda^2\mI \right\},
	 \end{align*}
	 where $\mH = \mC - \mI$ and $\ve$ is the resulting eigenvector. 
	   A solution of this optimization can be obtained, e.g., by solving
	   \begin{align*}
\hat{\lambda} =& \argmin \lambda \ \ \text{ s.t.}\\ &\left(\lambda \mI + \MMM{cc}{-2\mH & \mH^2 - {1\over \epsilon_{\vw}^2}\mH\vy\vy^T\mH \\ \mI & \Zer}\right)\MMM{c}{\lambda\ve\\\ve} =0.
	   \end{align*}

\begin{table}[t]		
	\caption{Update procedure for GALS and GALSR in the presence of noise }		
	\label{tb:aupdn}		
	\centering		
	\begin{tabular}{l}		
		\toprule
		\begin{minipage}{0.4\textwidth}
				\textbf{solve }$\wh = \argmin_{\|\vw\| = \epsilon_{\vw} } \|\vb_\sLam - \vgam_\sLam - R_{\LL} \OL \mM^\dagger \vw \|$\\ 
				 $\tilde\vb = \vb - \mO\mM^\dagger\wh $\\
				 $\tilde\vgam = \vgam - P_{\mL_\sLam}\mO\mM^\dagger\wh $\\
				 for (**) perform: $j \leftarrow \arg \min_{j\notin \sLam} |\tilde\vb_{(i)} - \tilde\vgam_{(i)}|^2 / (1+\valpha_{(i)} )$	\\
				 for (***) perform: $i \leftarrow \arg \max_{i\in \sLam} |\tilde\vb_{(i)} - \tilde\vgam_{(i)}|^2 / (1-\valpha_{(i)} )$
			
		\end{minipage}

		\\
		\bottomrule		
	\end{tabular}		
\end{table}	


\subsection{Analysis model numerical experiments} 

We turn to demonstrate empirically the performance of GALS and GALSR for synthetic signals and various images. 
\subsubsection{Synthetic signals}
\begin{figure}
	\centering
	\includegraphics[width=0.325\linewidth]{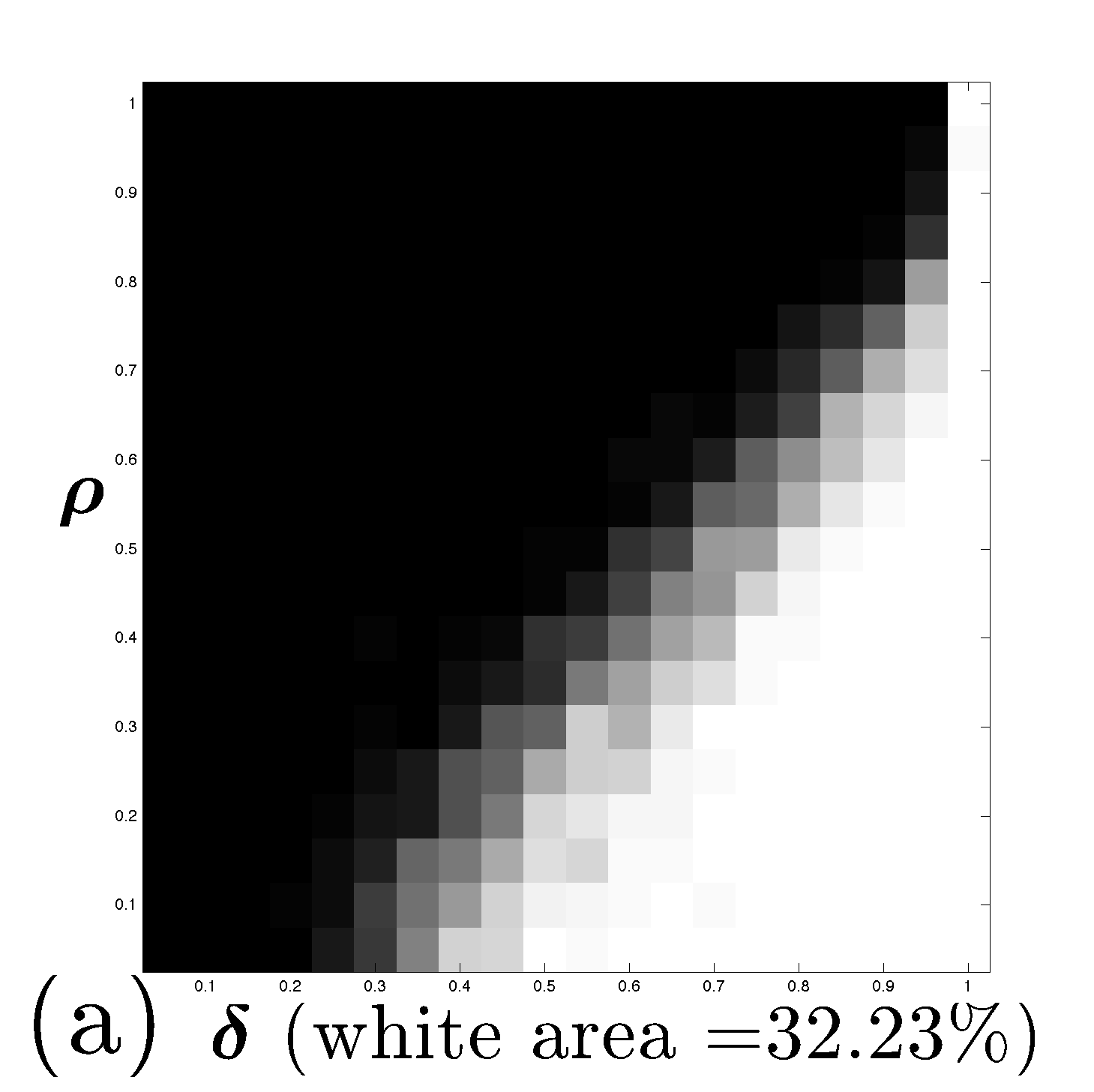}
	\includegraphics[width=0.325\linewidth]{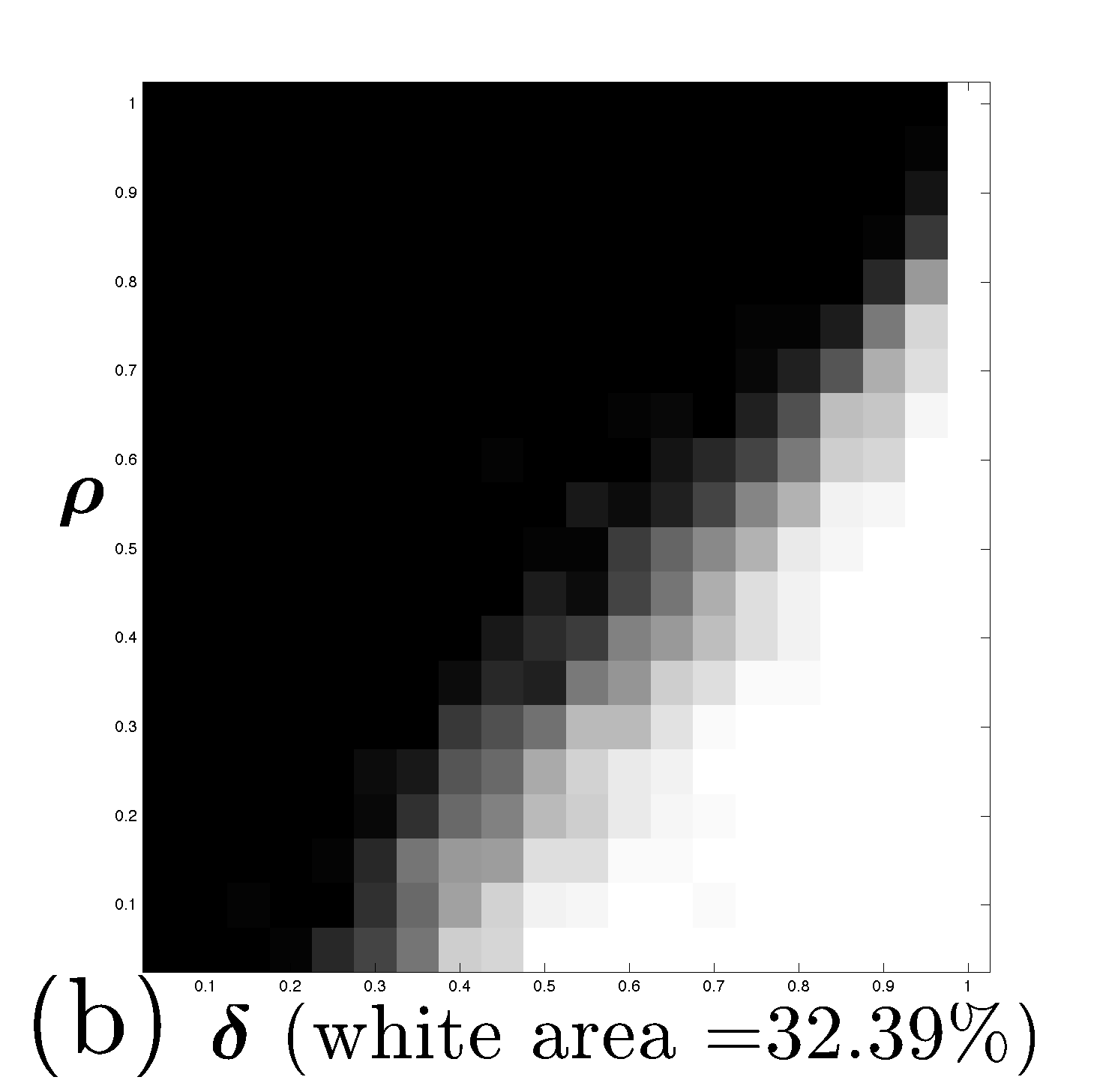}
	\includegraphics[width=0.325\linewidth]{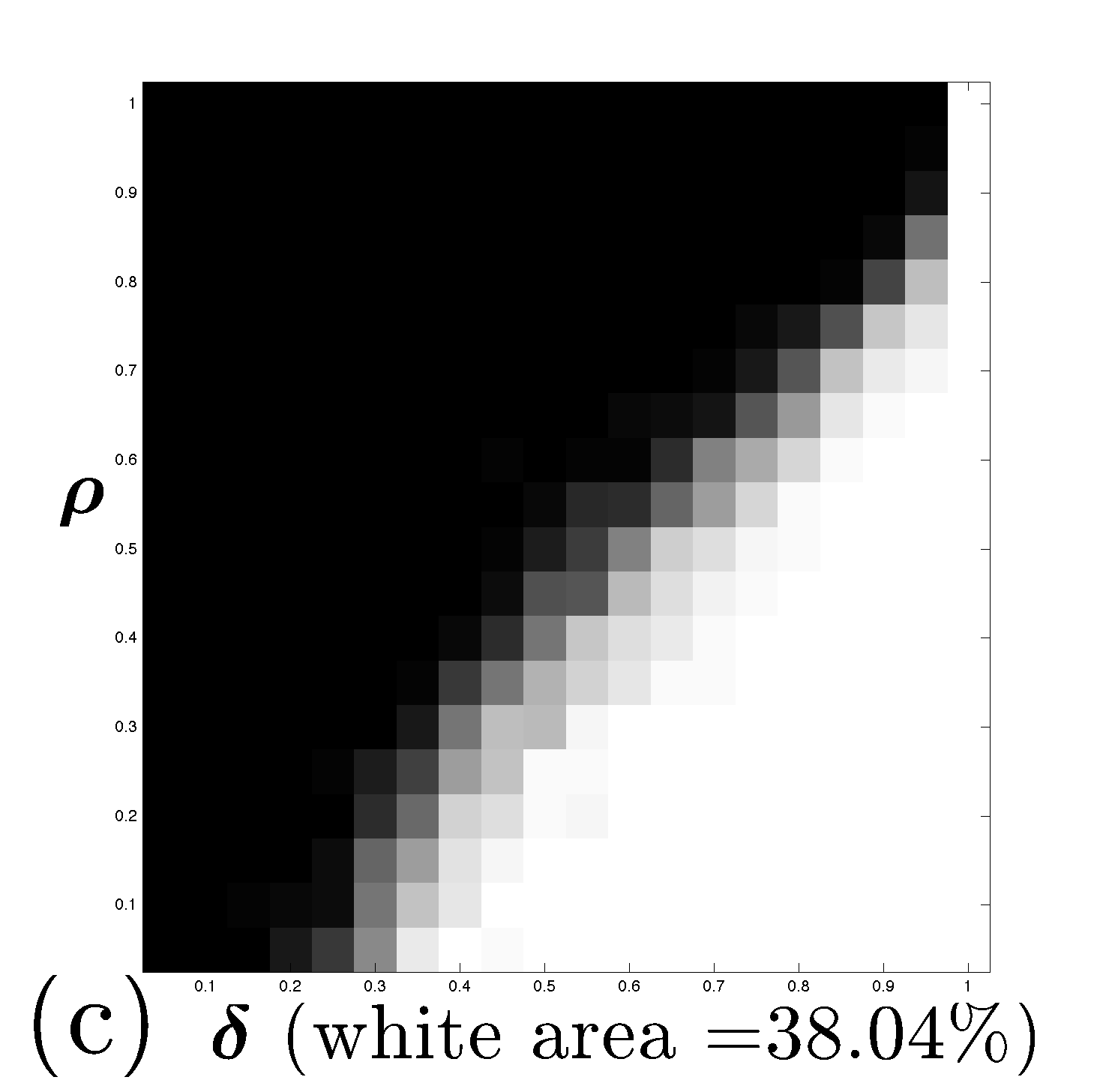}
	\caption[Analysis Experiments]{
		Phase transition of (a) GAP; (b) GALS; and (c) GALSR. 
	}
	\label{fig:aphase}
\end{figure}
 Following the experiment in \cite{giryes2014greedy}, we generate  $\mM$ to be with random Gaussian i.i.d. entries and $\mO$ to be a random Gaussian tight frame with $d = 120$ and $p = 144$. Figure \ref{fig:aphase} presents the phase transition diagram \cite{donoho2009counting} of GAP, GALS and GALSR\footnote{We refer the reader to \cite{giryes2014greedy} for a comparison to other methods in the same experiment. We have chosen GAP as it has given there the best results.}. The experiment is repeated 50 times for each pair of $m,l$. The gray level in each cell corresponds to the amount of times perfect reconstruction was achieved, i.e. white cells correspond to perfect reconstruction in all tests, and black cells correspond to $0\%$ success. It can be seen that GALSR reaches the largest white area in the phase diagram. 
\subsubsection{Shepp-Logan phantom reconstruction}
The second test is reconstructing the Shepp-Logan phantom. The results are presented in Fig. \ref{fig:ashepp}. The sampling operator is a sub-sampled two dimensional Fourier transform that measures only a certain number of radial lines in the Fourier domain of the image, and the analysis operator is the 2D difference operator. The stopping criteria for all algorithms is set to the actual cosparsity of the image under this operator ($l$ = 128014 for a 256$\times$256 size image). Fig.~\ref{fig:ashepp} (a) presents the original Shepp-Logan image; (b) shows the PSNR as a function of the amount of radial lines used in the image reconstruction in the noiseless case. In (c) AWGN is added to $\vy$ with varying energy ($x-$axis) and the PSNR is averaged over 10 realizations of the noise. For this part we use a smaller image size $64\times 64$ with 10 radial lines. Figures~\ref{fig:ashepp}(d)-(f) show the restoration results of GALS, GALSR, and GAP respectively from 30 radial Fourier lines of the 256$\times$256 Phantom image corrupted with AWGN with std of 3\% ($\|\vw\|/\|\vy_{0}\|$). The PSNR for all methods was slightly larger than 40db with a small advantage for GALSR.


\begin{figure}
	\centering
	\includegraphics[width=0.325\linewidth]{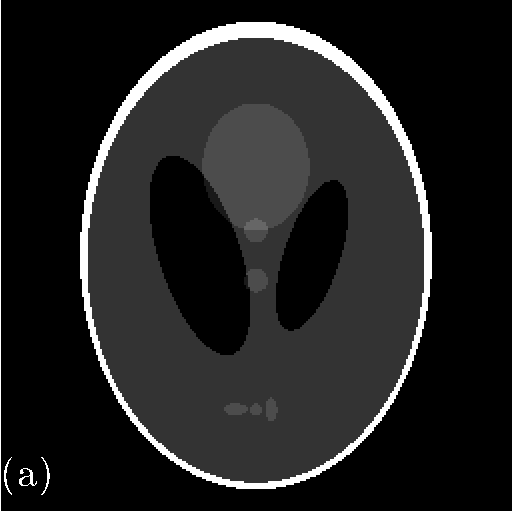}
	\includegraphics[width=0.325\linewidth]{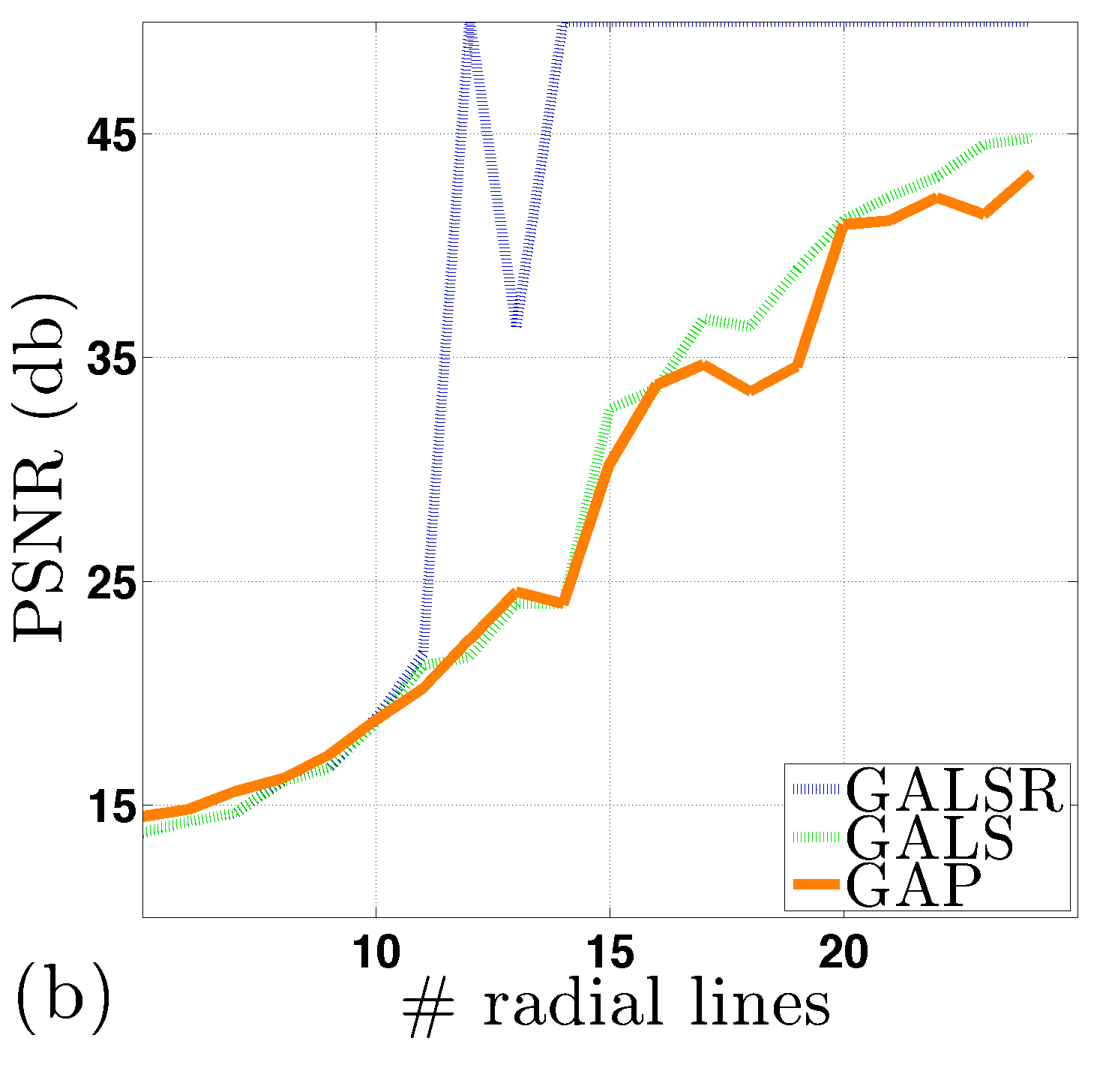}
	\includegraphics[width=0.325\linewidth]{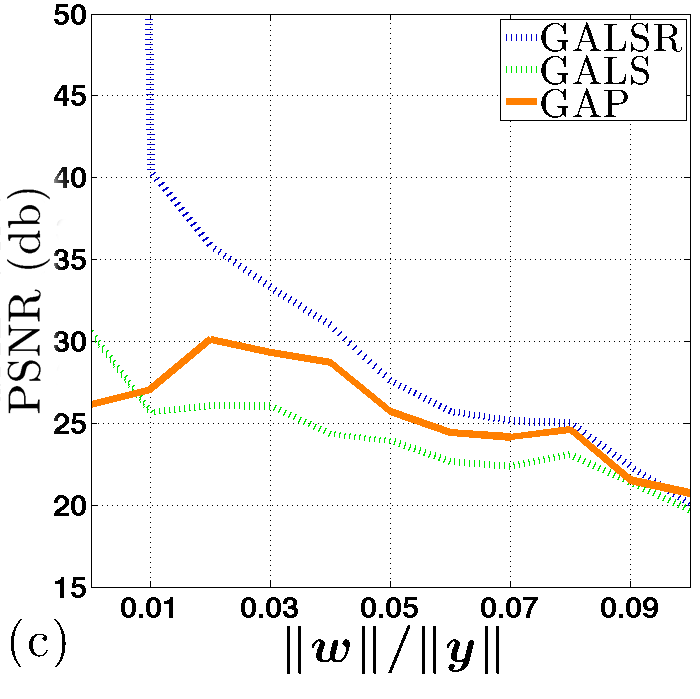}
	
	\includegraphics[width=\linewidth]{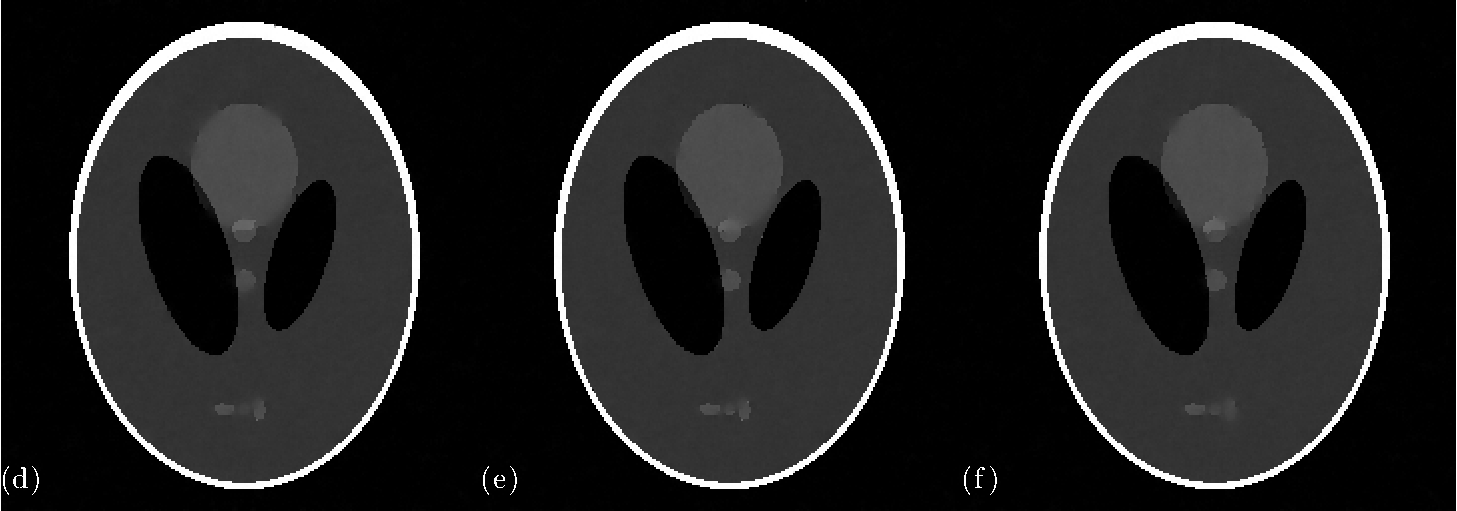}
	\captionsetup{singlelinecheck=off}
	\caption[Analysis Experiments]{Shepp-Logan reconstruction from a set of sub-sampled Fourier radial lines: (a) Input image of size $256\times 256$; (b) PSNR as a function of the number of radial lines used to sample the Fourier transform of (a) without noise. Results higher than 47db were truncated for display purposes (GLASR achieves perfect reconstruction as soon as the number of radial lines is larger than 14); (c) PSNR as a function of the AWGN level ($\|\vw\|/\|\vy_{0}\|$); (d) GALS reconstruction; (e) GALSR reconstruction; and (f) GAP reconstruction.  The PSNR for reconstructions (d)-(f) is slightly larger than 40db with a small advantage for GALSR.  
	}
	\label{fig:ashepp}
\end{figure}

\subsubsection{MRI Image Reconstruction}
The last experiment is related to the issue of having only an approximate signal model. Namely, the signals we meet in practice are not exactly cosparse. For this experiment, we have chosen (256 $\times$ 256 crop of) a MRI image generated from the FSL MNI152 T1 0.5mm image data \cite{jenkinson2012fsl}. The setting of the problem is similar to the Shepp-Logan case in the previous subsection, and the measurements of the image are obtained along 60 radial lines in the Fourier domain. However, no noise has been added. Fig.~\ref{fig:eee} presents the reconstruction results of GALS and GALSR in this case, which is better than the naive reconstruction of padding with zeros and then applying the inverse transform (25.58 dB). It can be seen that also in this case, where the cosparsity model is inexact, we get a good reconstruction using our methods. 

\begin{figure}
	\centering
	\includegraphics[width=\linewidth]{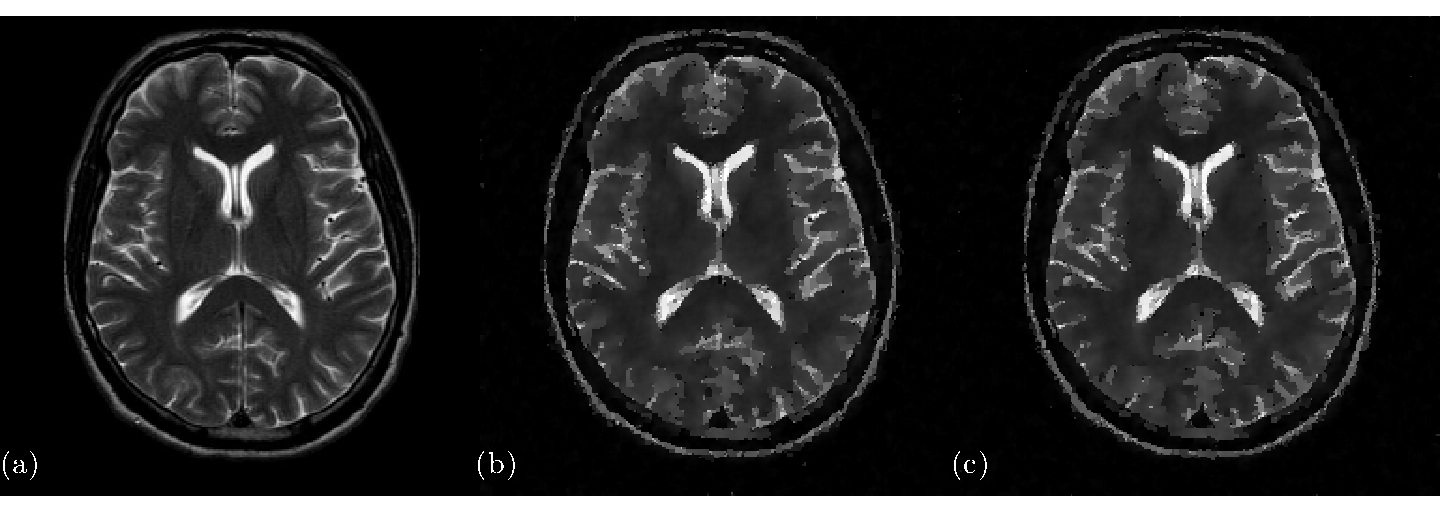}
	\caption[Analysis Experiments]{MRI image reconstruction: (a) original; (b) GALS (PSNR = 31.05 dB); and (c) GALSR (PSNR = 31.65 dB).}
	\label{fig:eee}
\end{figure}

\section{Conclusion} \label{sec:conc} 
In this work, we have presented two new pursuits for sparse recovery under the synthesis model, along with a theoretical study of their properties. We have shown that they provide a good alternative to OMP and OLS with better theoretical guarantees, which are similar (if not better) to the ones of SP and CoSaMP, while not requiring a prior knowledge of the size of the support (in the case of IOLSR). 
Both IOLSR and OLSR are simple to implement and demonstrate little overhead compared to OMP. While in this work, we have used the matrix inversion Lemma to provide an efficient implementation for them, other possibilities such as the QR or Cholesky decomposition exist, which may be also used to improve the computational efficiency \cite{Mairal10Online, Sturm12Comparison}. 

We posit that this kind of fast-converging strategy may 
be helpful in applications suffering from highly coherent dictionaries, as the correlation of an atom to the selected support is built into the algorithms (in the form of the vector $\vrho$) and does not require additional computations. 

We have also introduced two novel techniques for the analysis cosparse model: An OLS like algorithm, called GALS, and a version of it that incorporate replacements of atoms in the support, named GALSR. Both are provided with an efficient implementation. We have shown that in some applications this approach can be favorable compared to other existing strategies.

\section*{Acknowledgment}

\rg{We thank the anonymous reviewers for their valuable comments that significantly improved the presentation and structure of the paper.
This work was supported by the European research council (ERC-StG 757497 PI Giryes).
}

\bibliographystyle{plain}
{\bibliography{paper_refs}}

\appendix
\section{Proofs}
\label{sec:proofs}

\subsection{Proof of Lemma~\ref{lem:intro:bckgrnd:coladderr}}
\begin{proof}[Proof]
By using lemma \ref{lem:intro:bckgrnd:shermor}, we may write
\begin{eqnarray}
\label{eq:col_add:2}
\norms{R_\At\vy} & =& \norms{\vy-P_\At \vy} \\  \nonumber
&  \hspace{-0.5in} \overset{(a)}{=}  & \hspace{-0.3in}  \norms{\vy} - \vy^T\mA\mB\mA^T\vy -{1\over r}\left(\vy^T\Aa \left[ \begin{array}{c}\hat{\ve} \\ -1 \end{array} \right]\right)^2,
\end{eqnarray}
where $(a)$ follows from \eqref{eq:shermor} and some simple algebraic steps.  
Since $\|R_\mA \vy\|^2 = \norms{\vy} - \norms{P_\mA \vy} =\norms{\vy} - \vy^T\mA\mB\mA^T\vy $, we have, 
\begin{eqnarray}
\label{eq:bl23}
\|R_\At \vy\|^2 &=& \|R_\mA \vy\|^2 - {1\over r} \left(\vy^T (\mA\hat{\ve} -\va)\right)^2  \\ \nonumber  &=& \|R_\mA \vy\|^2 - (\vy^T\vv{R_\mA\va})^2.
\end{eqnarray}
Reordering \eqref{eq:bl23} together with the fact that $\|P_\mA \vy \|^2 = \|\vy\|^2 - \|R_\mA\vy\|^2$ leads to \eqref{eq:coll_add:resid}.
		\end{proof}	  
		
\subsection{Proof of Lemma~\ref{lem:intro:bckgrnd:xhatival}}
\begin{proof}Without loss of generality, we prove this formula for the last entry of $\hat{\vx}_{\mA}$ (i.e. $i=k$).  
Let $\tilde{\mA} = \Aa$ be with $k$ columns and consider the $k$-th entry of the expression 
				 $\hat{\vx}_\At = (\At^T\At)^{-1}\At^T\vy = \Bt\At^T\vy$.
				 Using Lemma \ref{lem:intro:bckgrnd:shermor} together with the notation therein we may rewrite $\hat{\vx}_{\tilde{\mA}}$ as (recall $\mA = \At_{\setminus k}$): 
				 \begin{align} \hat{\vx}_\At(k) = {1 \over \|R_{\mA} \va\|^2} ( \va^T - \va^T \mA \mB \mA^T )\vy,  
				 \nonumber
				 \end{align}
				 which equals \eqref{eq:xh_val} for $i=k$.
				 By permuting the entries of $\hat{\vx}$ and $\At$, the same can be proved for all $i \in [1,k]$.
\end{proof}

\subsection{Proof of Lemma~\ref{lem:intro:bckgrnd:leastcol}}
   \begin{proof}
   	From Lemma~\ref{lem:intro:bckgrnd:xhatival}, we have that $$\xh_\mA(i)^2 = {1\over \|R_{\Ai} \va\|^4} \langle \vv{R_\mA\va},\vy\rangle^2. $$
   	Combining this with the expression for $r$ in \eqref{eq:col_rem} that specifies the value of the diagonal of $\Bt$, we get that $ \langle \vv{R_\mA\va},\vy\rangle = \xh(i)^2/\tilde{\mB}(i,i)$. From   
   	Lemma \ref{eq:coll_add:resid} we have $\|R_\mA\vy\|^2 - \|R_\At\vy\|^2 =  \langle \vv{R_\mA\va},\vy\rangle$, which validates the claim.
   \end{proof}

\subsection{Proof of second equality in \eqref{eq:OL_xh_residual_norm} in Lemma~\eqref{alem:MSE}}
\label{sec:eq:OL_xh_residual_norm_proof}
We provide here the proof for the equality
\begin{eqnarray}
\label{eq:4AM}
\vy^T(\CL-\mI)\vy = \norms{R_{\LL} \OL \mM^\dagger \vy },
\end{eqnarray} 
\rg{where $R_{\LL} = \mI - P_{\LL}$ and $\LL = \QO^T\OL^T\in \RR^{(n-m)\times |\sLam|}$.}

\begin{proof}
In this proof, we will use the following identity (where $\sim$ mark cells that are not important for the proof): 
\begin{eqnarray}
\label{eq:4AM1}
\MMM{cc}{A & B \\ C & D}^{-1}  = \MMM{cc}{(A-BD^{-1}C)^{-1} & \sim \\ \sim & \sim }
\end{eqnarray}
Let $\mM = \matrx{U}\mS\QM^T$ be the SVD decomposition of $\mM$ where $\mS\in \RR^{\{m\times m\}}$ is a diagonal matrix with the  singular values of $\mM$, and $\matrx{U}\in\RR^{\{m\times m\}}$ and $\QM\in\RR^{\{n\times m\}}$ are the left and right singular-vector matrices respectively. 
From Lemma~\ref{lem:P1_xh}, $$\mC_\sLam = \left( \mM (\mM^T\mM + \OL^T\OL)^{-1}\mM^T \right)^{-1}.$$
Substituting in it
	$$ \mM^T = \mQ \MMM{c}{\QM^T\mM^T\\ \Zer},\ \ \ \OL^T = \mQ \MMM{c}{\QM^T\OL^T\\ \LL},$$ 
leads to
\begin{align}
 \mC_\sLam^{-1}  & =   \MMM{cc}{\mM\QM & \Zer}  \\ \nonumber & \cdot 
\MMM{cc}{\QM^T(\mM^T\mM + \OL^T\OL)\QM & \QM^T\OL^T\LL^T \\ \LL\OL \QM & \LL\LL^T} ^{-1}   \\ \nonumber & \cdot \MMM{c}{\QM^T\mM^T\\ \Zer}  ,
\end{align}
where $\mQ$ has canceled out due to its orthogonality. \rg{Notice that $\LL$ has full-row rank, which follows from Lemma~\ref{lem:B_S}.}\footnote{\rg{To see this, first note that $\QO$ has full rank and thus $\LL$ is not of full row rank only if $\exists \mathbf{v} \ne 0$ such that $\mathbf{u} =  \QO \mathbf{v}$ and $\OL \mathbf{u} =0$. By the definition of $\QO$, we have that $\mM\mathbf{u} = 0$. Thus, we get that $\mB_\sLam \mathbf{u} =0$ for $\mathbf{u}\ne 0$ in Lemma~\ref{lem:B_S}, which leads to a contradiction.}} 
Now, using 
\eqref{eq:4AM1}, we have
\begin{align*}
 \mC_\sLam^{-1}  & = \MMM{cc}{\mM\QM & \Zer}  \\ \nonumber & \cdot  \MMM{cc}{\left(\QM^T\mM^T\mM\QM + \QM^T\OL^T R_{\LL}\OL\QM \right)^{-1} & \sim \\ \sim & \sim} \\ \nonumber & \cdot \MMM{c}{\QM^T\mM^T\\ \Zer}  \\ \nonumber  & =   \mM\QM \left(\QM^T\mM^T\mM\QM  + \QM^T\OL^T R_{\LL}\OL\QM \right)^{-1}  \\ \nonumber &  ~~ \cdot \QM^T\mM^T, 
\end{align*}
where in the last equation, we have just opened the brackets. Noticing that $\mF \triangleq \mM\QM=\matrx{U}\mS$ is invertible, we have that 
\begin{align*}
\mC_\sLam &= \left(\mF \left(\mF^T\mF  + \QM^T\OL^T R_{\LL}\OL\QM \right)^{-1} \mF^T \right)^{-1}   \\
&= \mF^{-T} \left(\mF^T\mF  + \QM^T\OL^T R_{\LL}\OL\QM \right)\mF^{-1}.
\end{align*}
By opening the brackets, and then using the expression for $F$, we get 
\begin{align*}
\mC_\sLam &=  \mI  + \mF^{-T}\QM^T\OL^T R_{\LL}\OL\QM \mF^{-1}  \text{\ \ \ } \\
&= \mI  + \matrx{U}\mS^{-1} \QM^T\OL^T R_{\LL}\OL\QM\mS^{-1} \matrx{U}^{T}  \text{\ \ \ } \\
&= \mI  + \mM^{T\dagger}\OL^T R_{\LL}\OL\mM^\dagger  \text{\ \ \ },
\end{align*}
where in the last equality, we use the fact that $\mM = \matrx{U}\mS\QM^T$. 
To complete the proof subtract $\mI$ from the resulting expression and multiply by $\vy$ from both sides    
\end{proof} 

\subsection{Proof of Lemma~\ref{alem:errinc}}
\label{sec:alem:errinc_proof}
\begin{proof}
	Let $\alpha$, $\beta$, $\gamma$, and $\vq$ be as in the Lemma, and 
	denote $\GGi=(\LLi\LLi^T)^{-1}$ and $\vvvi = \OLi \mM^\dagger \vy$. 
	We first write the expression for $\norms{P_{\LLi}\vvvi}$ in terms of $\norms{P_{\LL}\vvv}$ and the constants above. By definition $P_{\LLi} = \LLi^T\GGi\LLi$. Thus, 
	\begin{align}
	& \hspace{-0.18in}	\norms{P_{\LLi}\vvvi} = \norms{\LLi^T\GGi\LLi\vvvi}\nonumber\\ 
		& \hspace{-0in} \overset{(a)}{=} \norms{\LLi^T(\GG - {1\over 1+\alpha} \GG\vq\vq^T\GG )(\LL\vvv +\beta\vq)}\nonumber\\
		& \hspace{-0in} \overset{(b)}{=}  \norms{\LLi^T(\GG\LL\vvv+\left(\beta - {\gamma \over 1 + \alpha} -  {\beta \rg{\alpha}\over 1+\alpha}\right)\GG\vq)}\nonumber\\
		& \hspace{-0in} \overset{c}{=} \norms{\LLi^T(\GG\LL\vvv+{\beta - \gamma\over 1+\alpha}\GG\vq)}\nonumber\\
		& \hspace{-0in} \overset{d}{=} \norms{\LL^T(\GG\LL\vvv+{\beta - \gamma\over 1+\alpha}\GG\vq)}\nonumber\\
		&+ \norms{\vq^T(\GG\LL\vvv+{\beta - \gamma\over 1+\alpha}\GG\vq)}\nonumber\\
	& \hspace{-0in}	\overset{}{=} \norms{\LL^T(\GG\LL\vvv+{\beta - \gamma\over 1+\alpha}\GG\vq)} + \left(\gamma + \alpha{\beta - \gamma \over 1 + \alpha}\right)^2\nonumber\\
& \hspace{-0in} \overset{e}{=} \norms{\LL^T(\GG\LL\vvv+{\beta - \gamma\over 1+\alpha}\GG\vq)} +{(\gamma + \beta\alpha)^2\over (1+\alpha)^2}\nonumber\\
	& \hspace{-0in}	\overset{f}{=} \norms{P_{\LL}\vvv} + 2{\beta - \gamma\over 1+\alpha}\gamma + {(\beta - \gamma)^2\over (1+\alpha)^2}\alpha + {(\gamma + \beta\alpha)^2\over (1+\alpha)^2}\nonumber\\
& \hspace{-0in}		\overset{g}{=} \norms{P_{\LL}\vvv} + \beta^2 - {(\beta-\gamma)^2\over 1+\alpha}, \label{aeq:vlue}
	\end{align}
where the following steps were used in the transitions: In step (a) we used the Sherman-Morrison formula on $\GGi = (\LL\LL^T + \vq\vq^T)^{-1}$ and $\LLi\vvvi = \MMM{cc}{\LL & \vq}\MMM{c}{\vvv  \\ \beta}  = \LL\vvv +\beta\vq$. For step (b) multiply the brackets and substitute the values for $\alpha$, $\beta$, and $\gamma$ in the appropriate places. Step (c) follows from scalar algebra. Step (d) can be derived by noticing that $\LLi^T$ has $|\sLam| + 1$ rows and that its squared-norm can be written as the sum of the first $|\sLam|$ rows plus the last row (\rg{namely, $\vq^T$},
assuming, w.l.o.g., that $i$ is the last entry in $\OLi$). 
\rg{In step (e), we use again the definition of $\alpha$ and $\gamma$ and the fact that a norm of a scalar is just its square value.}
Step (f) follows from opening the remaining norm in the expression. Finally step (g) is the result of the following expansion (set the common denominator to be $(1+\alpha)^2$ and look at the nominator):
\begin{align*}
&2\gamma(\beta-\gamma)(1+\alpha) + (\beta-\gamma)^2\alpha + (\gamma+\beta\alpha)^2\\
&\ \ = 2\gamma(\beta + \beta\alpha -\gamma - \gamma\alpha) + \alpha(\beta^2 -2\beta\gamma+\gamma^2)  \\ &\ \ + \gamma^2 +2\alpha\beta\gamma + \beta^2\alpha^2\\
&\ \ = 2\gamma\beta -\gamma^2-\gamma^2 \alpha + \alpha\beta^2 +2\gamma\beta\alpha+\beta^2\alpha^2 \\
&\ \ = (1+\alpha)^2\beta^2 - \beta^2\alpha -\beta^2 +2\gamma\beta - \gamma^2 - \gamma^2\alpha +2\gamma\beta\alpha \\
&\ \ = (1+\alpha)^2\beta^2 - (\beta^2- \gamma^2)(1+\alpha) 
\end{align*}

	Now, to conclude the proof, expand the expression for the error from Lemma \ref{alem:MSE} for $\sLam\cup i$ and combine with \eqref{aeq:vlue}: 
	\begin{align*}
		\norms{\OLi\xh_2}  &=  \norms{R_{\LLi} \OLi \mM^\dagger \vy } \\
		&= \norms{\vvvi } -  \norms{P_{\LLi}\vvvi}\\
		&= \norms{\vvv}+\beta^2 - \norms{P_{\LLi}\vvvi}\\
		&= \norms{\vvv}+ \norms{P_{\LL}\vvv} + {(\beta-\gamma)^2\over 1+\alpha}\\
		&=\norms{\OL\xh_1}  + {(\beta-\gamma)^2\over 1+\alpha}
	\end{align*} 
\end{proof}

\end{document}